\documentclass[mnsc,nonblindrev]{informs3_clean}
\OneAndAHalfSpacedXI % current default line spacing
\usepackage[hypertexnames=false, colorlinks = true, linkcolor = red, citecolor = red]{hyperref}
\usepackage{endnotes}
\let\footnote=\endnote

\usepackage{multicol}
\usepackage[english]{babel}
\usepackage{multirow}
\usepackage{color}
\usepackage{graphicx}
\usepackage{lmodern}
\usepackage{anyfontsize}

\usepackage{algorithm}% http://ctan.org/pkg/algorithms
\usepackage{algpseudocode}% http://ctan.org/pkg/algorithmicx

\usepackage{enumitem}
\usepackage[pagewise]{lineno} %displaymath,  mathlines
\usepackage{subcaption}
\usepackage{longtable}
\usepackage{booktabs} 
\usepackage{blkarray}% http://ctan.org/pkg/blkarray

\newcommand{\R}{\mathbb{R}}

\newcommand{\E}{\mathbb{E}}

\newcommand{\ignore}[1]{}

\newcommand{\bbf}{\mathbf f}

\newcommand{\bH}{\mathbf H}
\usepackage{eurosym}
\usepackage{titlesec}
\titleformat{\paragraph}[runin]{\bf}{}{}{}[]

\usepackage{xpatch}
\algrenewcommand\alglinenumber[1]{{\sffamily\footnotesize#1}}
\makeatletter
\xpatchcmd{\algorithmic}{\itemsep\z@}{\itemsep=-2ex plus0pt}{}{}
\makeatother
\newcommand{\bs}{\mathbf s}

\newcommand{\bR}{\mathbf R}
\newcommand{\bomega}{\boldsymbol{\omega}}
\newcommand{\bX}{\mathbf X}
\newcommand{\bx}{\mathbf x}

\newcommand{\static}{\mathsf{STATIC}}

\newcommand{\cont}{\mathsf{c}}

\newcommand{\Be}{\mathbf e}

\newcommand{\bzero}{\mathbf 0}

\newcommand{\osp}{\operatorname{sp}}
\usepackage{soul}
\let\originalparagraph\paragraph
\renewcommand{\paragraph}[2][.]{\originalparagraph{#2#1}}

\newcommand{\hypo}{\textrm{Hypo}}
\providecommand{\norm}[1]{\lVert#1\rVert}
\def\defi{:=}

\graphicspath{{figures/}}

\allowdisplaybreaks
\usepackage{natbib}
 \bibpunct[, ]{(}{)}{,}{a}{}{,}%
 \def\bibfont{\small}%
\TheoremsNumberedThrough     % Preferred (Theorem 1, Lemma 1, Theorem 2)
\ECRepeatTheorems

\EquationsNumberedThrough    % Default: (1), (2), ...
\begin{document}
%%%%%%%%%%%%%%%%

%\setcounter{page}{1}
%\DOI{10.1287/xxxx.0000.0000}%

 \RUNAUTHOR{Suk and Wang}
%\RUNAUTHOR{Blinded authors}

\RUNTITLE{Optimal Pricing for Tandem Queues}

% \TITLE{Bundling Information Goods of Decreasing Value}
\TITLE{Optimal Pricing for Tandem Queues: Does It Have to Be Dynamic Pricing to Earn the Most}

\ARTICLEAUTHORS{%
 \AUTHOR{Tonghoon Suk}
 \AFF{Mathematical Sciences Department, IBM Thomas J. Watson Research Center, Yorktown Heights, NY 10598, \EMAIL{tonghoon.suk@ibm.com}} %, \URL{}}

 \AUTHOR{Xinchang Wang}
 \AFF{Department of Marketing, Quantitative Analysis, and Business Law, Mississippi State University, Mississippi State, MS 39762, \EMAIL{xwang@business.msstate.edu}\\ \today} %, \URL{}}
%\AUTHOR{Marg Arinella}
%\AFF{Institute for Food Adulteration, University of Food Plains, Food Plains, MN 55599, \EMAIL{m.arinella@adult.ufp.edu}}
} % end of the block

\ABSTRACT{%
% Enter your abstract
Tandem queueing systems are widely-used stochastic models that arise from many real-life service operations systems. Motivated by the desire to understand the trade-off between the performance and  complexity of policies for capacity-constrained tandem queueing systems, we investigate the long-run expected time-average revenue, the gain, of the service provider for various pricing policies. The gain-maximization problem is formulated as a Markov decision process model 
but the optimal policy, which dynamically adjusts service prices, is hard to obtain due to the curse of dimensionality.
For general tandem queueing systems, rather than identifying an optimal dynamic policy, we show that the best possible static policy that quotes the same price to all customers is asymptotically optimal when the buffer size at the first station is sufficiently large. A noteworthy feature of our analysis is that we identify an easy-to-obtain but asymptotic optimal static policy associated with a simple optimization problem. We validate our analytic results through numerical experiments and learn that, surprisingly, the gain under the simple static policy is close to the optimal gain even when the buffer size at the first station is moderate.

}%

%\KEYWORDS{deterministic inventory theory; infinite linear programming duality;
%  existence of optimal policies; semi-Markov decision process; cyclic schedule}

%\HISTORY{This paper was
%first submitted on April 12, 1922 and has been with the authors for
%83 years for 65 revisions.}

\maketitle
%%%%%%%%%%%%%%%%%%%%%%%%%%%%%%%%%%%%%%%%%%%%%%%%%%%%%%%%%%%%%%%%%%%%%%

% Samples of sectioning (and labeling) in OPRE
% NOTE: (1) \section and \subsection do NOT end with a period
%       (2) \subsubsection and lower need end punctuation
%       (3) capitalization is as shown (title style).
%
%\section{Introduction.}\label{intro} %%1.
%\subsection{Duality and the Classical EOQ Problem.}\label{class-EOQ} %% 1.1.
%\subsection{Outline.}\label{outline1} %% 1.2.
%\subsubsection{Cyclic Schedules for the General Deterministic SMDP.}
%  \label{cyclic-schedules} %% 1.2.1
%\section{Problem Description.}\label{problemdescription} %% 2.

% Text of your paper here

\section{Introduction}\label{sec:introduction}
Many real-life stochastic service systems can be modeled as tandem queueing systems with finite buffers. One example in \cite{altiok:00} is the container processing system at a seaport's terminal. An outbound container passes through serial service processes in a row before it reaches a slot on a departing ship, such as being unloaded at the container yard, hauled to the working area of a quay crane, and loaded to the ship by the quay crane. An inbound container undergoes an exactly reverse journey from getting off an arriving ship to being placed at the container yard. %Another example is a chain of interconnected attraction sites within a tourism zone, where two or more sites are spatially connected so that visitors leaving one site find themselves at the entry point of the next one, such as the attractions in the World of Coca-Cola at Atlanta.
Other examples include call centers, hospital emergency rooms, cost-effective blood screenings, and wireless networks (cf.~\cite{Dijk:88, Le:08, Bar-Lev:13}). %\tadd{The topic of this paper is focused on tandem queueing systems.} 

%% Xinchang's comment: This sentence is way too complicated to understand. If we have references, do not have to tell everything. Refer readers to references.
% \tdel{Other examples include 
%a chain of interconnected attraction sites within a tourism zone, where two or more sites are spatially connected so that visitors leaving one site find themselves at the entry point of the next one, such as the attractions in the World of Coca-Cola at Atlanta, call centers, where customers are transfered to specialists after taking to general call takers, 
%hospital emergency rooms, where patients see triage nurses and go to a number of medical test and procedures, cost-efficient blood screenings, where expensive test is taken after fast and inexpensive test is conducted, and wireless networks, where packets processed in a access point (AP) or terminal moves to another AP or terminal (e.g, see \cite{Dijk:88, Le:08, Bar-Lev:13}}.

\paragraph{Entry control}
The service provider controls the entries of customers into a queueing system for specific reasons, such as congestion mitigation \citep{banerjee:12} or revenue maximization \citep{ziya:06}. 
Our work particularly focuses on revenue management for tandem queueing systems,  in which the service provider (he) controls customers' entries through pricing in order to maximize the earned revenue. More specifically, the service provider announces a price to arriving customers. An incoming customer (she) pays the price and enters the system if the price is less than or equal to her \emph{reservation price} (i.e., the highest willingness-to-pay price); otherwise, the customer leaves without purchasing the service. An entrant, which we call an \emph{actual customer}, leaves the system only after completing jobs at all the stations in the tandem system. Customers who wait for service at a station are stored in a finite-sized buffer in front of the station. Meanwhile, an upstream station is allowed to start service only if the buffer space before the downstream station next to it is not fully occupied, which is called the \emph{Communication Blocking Mechanism} \citep{cheng:93}.

\paragraph{Revenue and control policy} The service provider earns revenue by collecting payments from customers. We  assume that the operating cost for the  system is fixed and independent of the quantity of customers in the system. 
This assumption is not restrictive in practical applications, in which the variable operating cost per customer is marginal compared with the cost that has to be paid regardless of the number of customers in the system, such as the installation cost, fixed operating cost, and crew salaries, and is widely used in the literature of airline revenue management \citep{tall:04b} and queue revenue management \citep{afeche:16}. 
In this study, the assumption also allows us  to obtain  fundamental results with simplicity. Without loss of generality, we assume that the operating cost is zero.

In this setting, an optimal pricing problem arises: The service provider needs to determine an optimal pricing policy that pre-specifies the price quoted for each state  of the system (e.g., the combination of queue lengths at all stations) to maximize a system performance criterion such as the long-run expected time-average revenue or  infinite-horizon discounted  revenue. The most favorable policy best balances the trade-off between setting a high price to take advantage of a high marginal revenue  and setting a low price to benefit from a high volume of service buyers.
In this paper, we focus on (approximate) optimal policies maximizing the system's long-run expected time-average revenue, which is referred to as the \emph{gain}, from the perspective of a service provider.

\paragraph{Dynamic Versus Static Pricing} The provider could employ one of the two pricing schemes: dynamic or static pricing. 
While the former adjusts quoted prices dynamically based on the state of the system, the latter quotes the same static price to all customers. Obviously, an optimal dynamic pricing policy achieves a revenue no worse than that obtained under any static pricing policy. However, finding an optimal dynamic policy requires enormous computational efforts for a multi-station tandem queueing system due to the curse of dimensionality.  On the other hand, the static pricing problem has a far smaller size than the dynamic pricing problem and a static  policy is easier to implement than a dynamic policy.%, thus \tadd{providing favorable solutions to practitioners}. 

At this point, our motivation is answering the following questions:
\begin{itemize}
\item Can the best static pricing policy, hereafter called an \emph{optimal static policy},  perform as well as an optimal dynamic pricing policy for a finite-buffered tandem queueing system?
\item
If so, under what conditions and how can we find a well-performing static policy?
\end{itemize}
For a system with a capacity-constrained single station and multiple user classes,  it has been proved that the optimal static policy performs as well as the optimal dynamic policy under the condition that the system resource (e.g., bandwidth in a communication network) tends to infinity \citep{paschalidis:00}. However, whether or not the same or similar result holds for a tandem queueing system (which has heavier inter-dependences between stations)  remains yet unknown.

\subsection{Our Contributions}\label{sec:our contributions}
Assuming that the reservation prices of customers are independent and identically distributed, the arrival process of customers follows a (homogeneous) Poisson process, and service times at each station are independent and exponentially distributed, we formulate the optimal pricing problem as a unichain Markov decision process (MDP) model (see Section~\ref{sec: the model} for more details) with the objective to maximize the system's gain. 
%% Xinchang's comment: the term gain has before been defined.

\paragraph{A motivating example} We reach out to a numerical example to obtain  initial insights into the performances of optimal static and dynamic pricing policies for a tandem queueing system. In the numerical example, we choose a tandem line of two stations, in which service times at both stations are exponentially distributed with rate $8.0$ and the arrival rate of customers is $3.6$. We set the size of the buffer at station $2$ to be $5$ (i.e., the maximum queue length at station 2 is 6, including the one being served) and allow the buffer size at station 1 to change. The service provider quotes prices from among a discrete set of prices, $\{350,400,450,500, 550, 600,650,700, 750\}$. We hereby imitate these prices as the container handling fees at a seaport; for example, it could cost from $\text{\euro}  260$  to $\text{\euro} 1,014$  to handle a container at EUROGATE Container Terminal Hamburg \citep{eurogate:14}.

%and the buffer size before the second station is $0$.}

% customers arrive according to a homogeneous Poisson process with rate $3.6$, the service time at station $1$ is exponentially distributed with rate $8.0$, and the service time at station $2$ follows an exponential distribution with rate $8.0$. We fix the buffer size in front of station 2 to be $5$ and allow the buffer size in front of station 1 to change.
% The service provider quotes prices from among a discrete set of prices, $\{350,400,450,500, 550, 600,650,700, 750\}$. We hereby imitate these prices as the container handling fees at a seaport; for example, it could cost from $\text{\euro}  260$  to $\text{\euro} 1,014$  to handle a container at EUROGATE Container Terminal Hamburg \citep{eurogate:14}.

We use the unichain policy iteration algorithm \citep{puterman:94} to find an optimal dynamic policy for the MDP model. We then evaluate the gain under each  price and choose the price resulting in the maximum gain as the optimal static price. Figure~\ref{fig: zero holding cost exp} shows  the optimal gains as a function of the buffer size at station $1$ under static and dynamic pricing schemes, when the reservation prices of customers are exponentially distributed with rate 0.002. Figure~\ref{fig: zero holding cost uni} shows the same content with uniformly distributed reservation prices in $[500, 1200]$. Figure~\ref{fig: the change of optimal long-run average profit with B1} suggests that the gap between the gain under the optimal dynamic pricing policy and that under the optimal static pricing policy  gradually vanishes as the buffer size before station 1 becomes large.

\begin{figure}[ht!]
\centering
\begin{subfigure}[b]{0.45\textwidth}
\centering
\includegraphics[width=\textwidth]{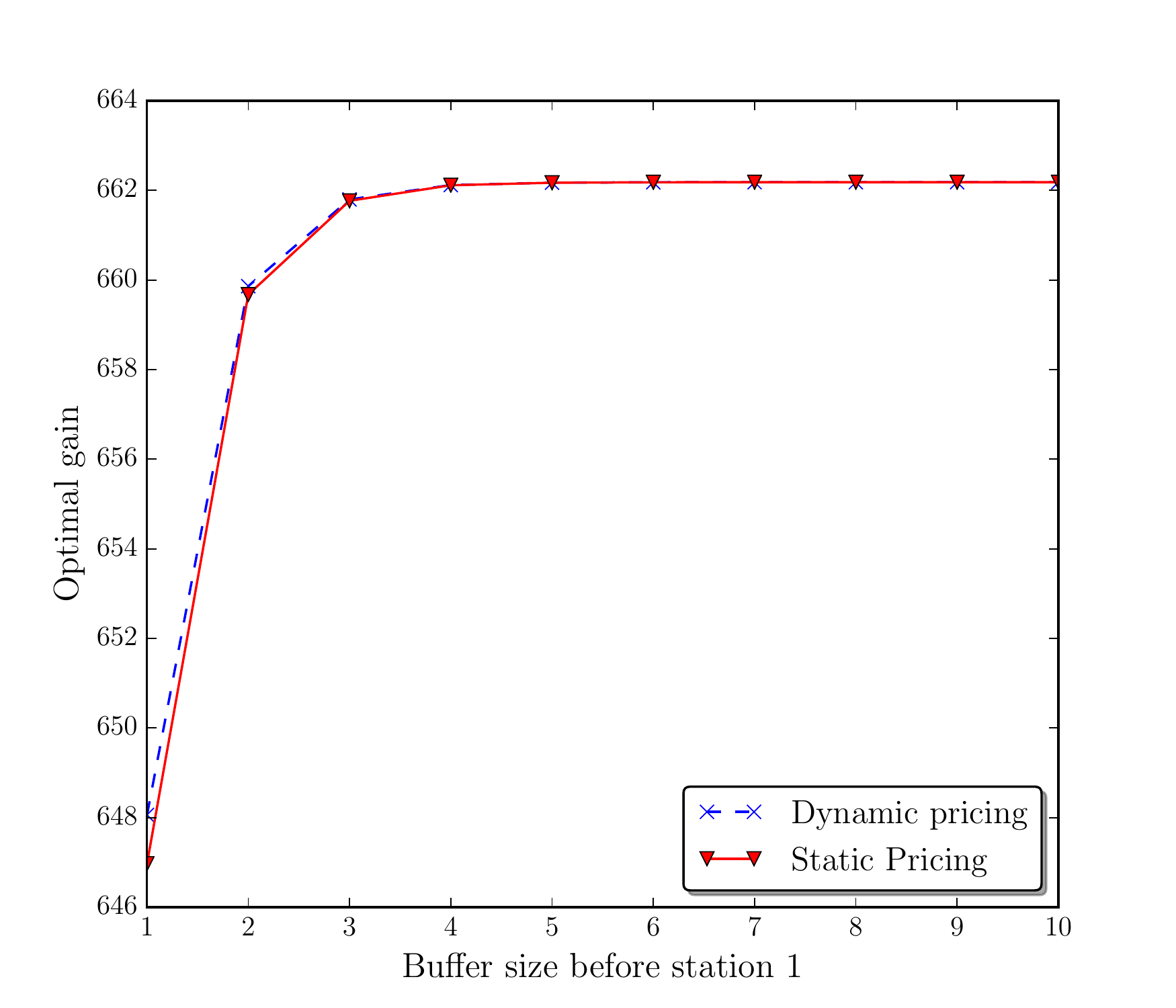}
\caption{Reservation prices $\sim$ Exponential($0.002$)}
\label{fig: zero holding cost exp}
\end{subfigure}
\begin{subfigure}[b]{0.45\textwidth}
\centering
\includegraphics[width=\textwidth]{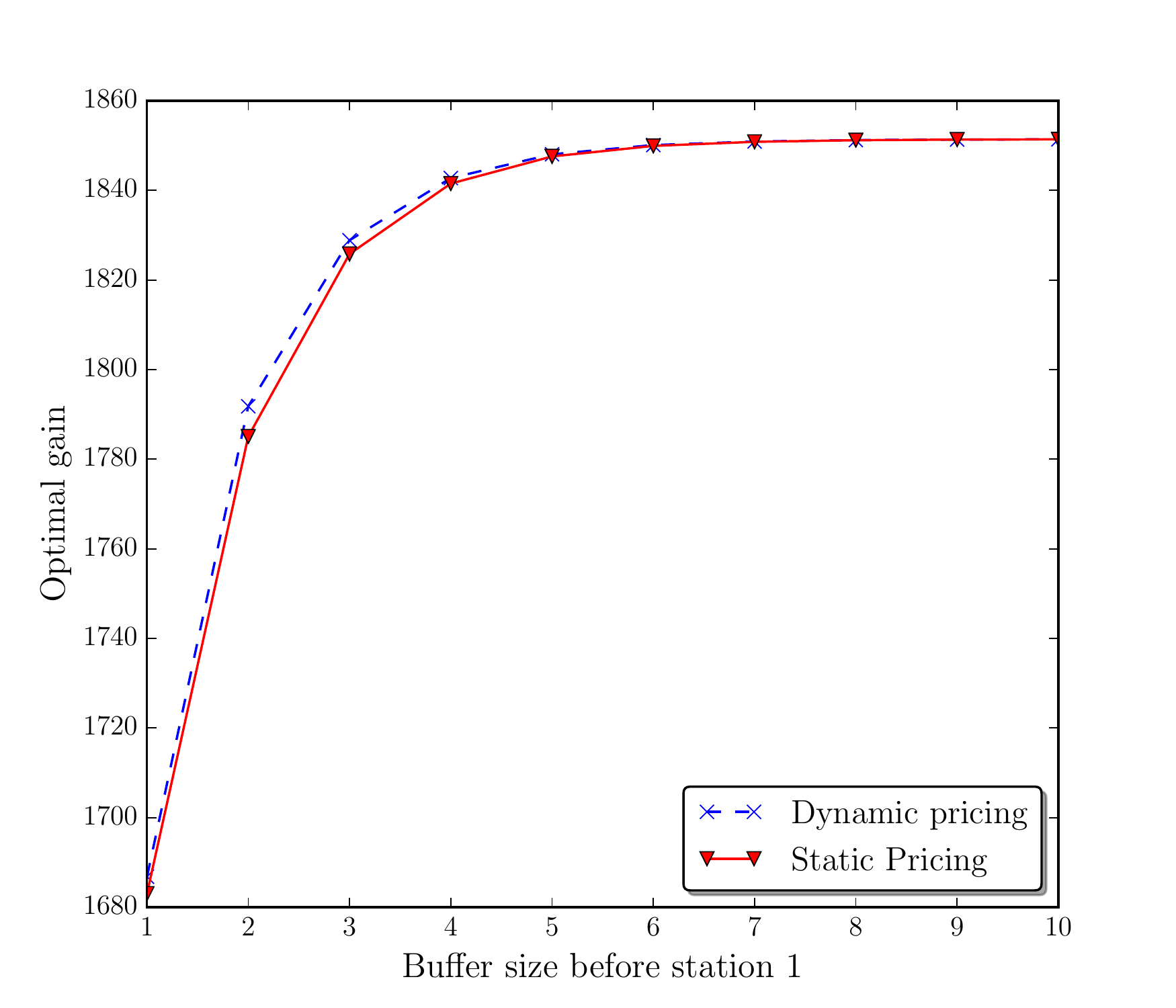}
\caption{Reservation prices $\sim$ Uniform[500, 1200]}
\label{fig: zero holding cost uni}
\end{subfigure}
\caption{Optimal gains under dynamic and static pricing schemes.}
\label{fig: the change of optimal long-run average profit with B1}
\end{figure}

With what we have learned from the literature (e.g., \citet{paschalidis:00}) and the example, this work is aimed at investigating the optimal pricing policies for a tandem queueing system with finite-buffered stations. More specifically, the objective of this paper is as follows: For a general tandem queueing system with arbitrarily  many stations, we aim to understand under what conditions the optimal static pricing performs as well as the optimal dynamic pricing.

\paragraph{Summary of main contributions}
Our work makes  both theoretical and managerial contributions to the literature on revenue management for queues:

%\paragraph{Theoretical contributions}
\begin{enumerate}[itemsep=0in, label=(\arabic*)]
\item
For a general tandem queueing system with an arbitrary number of stations and a finite buffer in front of each station, we show that the optimal static policy is asymptotically optimal (over the set of all pricing policies). 
More specifically,  the gap between the gain obtained from an optimal static policy and the optimal gain converges to zero as the size of the buffer  before the first station tends to infinity. 

\item 
We show that the convergence of the optimal static gain to the optimal gain occurs at an exponential rate.
%\item
%For a two-station tandem queueing system, in which the service provider quotes from a high and a low price and the buffer size for the second station is zero, we characterize the necessary and sufficient conditions that guarantee the optimality of static pricing policies.   Meanwhile, we provide conditions under which the static policies are uniquely optimal. This result provides the service provider with conditions used to validate the optimality of static policies without using a numerical method such as the policy iteration algorithm.
%\item
%We also characterize the optimal dynamic policies for the two-station tandem line, by which we are able to cut off the number of possibly optimal dynamic policies from an exponentially large value to a polynomial number. 

\item
When the buffer in front of the first station is large, it is sufficient and convenient to solve a relatively easy static pricing problem rather than  a hard dynamic pricing problem.
\item
When the buffer  before the first station is large, the service provider does not even need to solve the static pricing problem itself, particularly when doing so is quite hard with a continuous set of prices. Instead, she can obtain a good approximate solution by solving a much simpler static optimization model (i.e., upper-bound optimization model~\eqref{eq:ideal_gain_problem}).
%\item
%When the service provider is managing a two-station system as described above, there exists only three types of pricing policies that could possibly be  optimal and the number of these three types of policies is a linear function of the buffer size of the first station.
%\item
%For the two-station system, it tends to be optimal to quote the low price to all incoming customers (to motivate customers to enter the system) if the marginal gain (i.e., the change in gain over the change in the quantity of customers entering the system) is high; otherwise, it tends to be optimal to quote the high price to all incoming customers (to discourage customers from entering the system). 
\end{enumerate}

%% literature review

\subsection{Literature Review}
\label{sec: literature review}

A great deal of literature on revenue management for queues, mainly through price optimization, exists and most of the studies are focused on single-station queues or parallel queues. 

\paragraph{Revenue management for single-station queues} 
\citet{naor:69} was one of the first studies addressing the static pricing problem for a single-station queue with a finite buffer. \citet{ziya:06,ziya:08} focused on the static pricing for a $G/GI/s/m$ queue and derived analytical  expressions of the optimal static prices for $M/M/1/m$ and Erlang loss systems. \citet{maoui:09} analyzed optimal static pricing policies that maximize the long-run average profit for $M/G/1$ and $M/M/1/N$ queues with holding cost.   \citet{haviv:14} evaluated the performance of a static demand-independent price  for an $M/M/1$ queue using a robust optimization model. The authors found an interesting result that the optimal demand-independent price can perform well compared with the optimal static price obtained with full knowledge of the arrival rate. In a recent paper, \citet{afeche:16} solved the joint price and lead-time quotation problem for an $M/M/1$ queue with multi-type customers and was focused on characterizing the optimal solution to a static nonlinear revenue maximization model with mechanism design constraints under customer choice. \citet{hassin:17} studied a high-low dynamic pricing mechanism for an M/M/1 queue, where a high and a low price are quoted based on the number of customers in the system, and particularly, the authors quantified the profit loss when the queue manager adopts a static price and the FCFS rule.

\citet{low:74a} formulated the dynamic pricing problem as an MDP model for an $M/M/s/N$ queue. The study was later extended for the same queue with an infinite buffer  in \citet{low:74b}. \citet{yoon:04} considered both pricing and admission control problems for a single-station queueing system with multi-class customers and time-dependent arrival and service rates. \citet{maglaras:06} considered the dynamic pricing problem for a single-server queue with multi-class customers and infinite buffers and developed a fluid approximation method to solve the problem rather than solving a computationally demanding MDP model. %\citet{maoui:07} examined dynamic pricing for a queueing system with multi-class customers and finite/infinite capacity. 
\citet{cil:11} investigated a single-station queue with two classes of customers and showed that the optimal dynamic pricing policy is of a monotone structure in the queue lengths for the two customer classes. 
\citet{afeche:13} studied the Bayesian dynamic pricing problem for the $M/M/1$ queue, in which the service provider chooses to quote a high or low price, or reject customers to maximize the system's revenue, while using Bayesian updating to learn the distribution of patient and impatient customers.

%{There exists a}
Another stream of studies exists (e.g., \citet{aktaran:09} and \citet{cil:09}) with a particular interest in analyzing the sensitivity of the optimal dynamic pricing policies to various system parameters, such as the number of servers, arrival and service rates, and buffer sizes.

\paragraph{Revenue management for a network of queues}
The literature on revenue management for a network of queues is sparse, especially for those with finite buffers. 
 \citet{ghoneim:85} considered an admission control problem for arrivals to a tandem line of two queues with Poisson arrivals, exponential service times, and infinite buffer sizes. The authors formulated the problem as a discounted MDP model. \citet{ching:00} studied a tandem line of two uncapacitated queues. Due to the unlimited buffer sizes, the two queues can be treated as two separate systems and the authors determined the optimal static pricing policies based on the steady-state probabilities for each queue. 

\citet{paschalidis:00} is one of the works close to and inspiring ours, in which the authors showed that the optimal static policy can perform as well as the optimal dynamic policy for a single-station queue with multi-class users when the system resource (i.e., capacity or buffer) becomes unlimited.  %\citet{wang:16} studied the performance gap between optimal static and dynamic pricing policies particularly for two-station tandem queueing systems, while this work deals with general tandem queueing systems that have arbitrarily many stations.

\paragraph{Positioning}
Our work continues with the existing literature on revenue management for queues and discusses an interesting question as to whether or not the static pricing can achieve as much gain as the dynamic pricing does for a general tandem queueing system, and if so, it would suggest  practitioners to adopt the much simpler  static pricing scheme for revenue maximization.

\subsection{Outline of This Paper}
\label{sec:Outline of This Paper}
The plan of the remaining paper is as follows. Section~\ref{sec: the model} describes the optimal pricing problem and  formulates the problem as an MDP model.  We state our main result in Section~\ref{sec:summary of main results} for readers to have a quick access to the major contributions made in this work. 
Section \ref{sec:numerical experiments} contains numerical experiments that test the performance of static policies.
We provide in Section~\ref{sec:proof_lower_bound} the detailed elaborations and proof for the main result. Some lengthy proofs and technical details are provided in the e-companion.

%%%%%%%%%%%%%%%%%%%%%%%%%%%%%%%%%%%%%%%%%%%%%%
%							The Model                   %
%%%%%%%%%%%%%%%%%%%%%%%%%%%%%%%%%%%%%%%%%%%%%%

\section{Systems Description and MDP Formulation}
\label{sec: the model}
This section introduces tandem queueing systems, in which the service provider quotes prices to price-sensitive customers and offers service.  For the systems, we define an optimal pricing problem and describe it as a Markov decision process (MDP) model.

\subsection{Notation}
\label{sec:notation}
We denote the set of non-negative real numbers, the set of non-negative integers, and the set of natural numbers by $\mathbb{R}_{+}$, $\mathbb{Z}_{+}$, and $\mathbb{N}$, respectively. For $n \in \mathbb{Z}_{+}$, $[n]$ represents the set of all non-negative integers less than or equal to $n$ (i.e., $[n]:=\{0,1,2,\dots,n\}$). 
% \tdel{For convenience, we let $[\infty]=\mathbb{Z}_{+}$.}[\xw{Not USED!}]  
The cardinality of a (finite) set $\mathcal{A}$ is denoted by $|\mathcal{A}|$ (e.g., $|[n]|=n+1$).
For a right-continuous function $f:\mathbb R\to\mathbb R$ with left limits and $a\in\mathbb R$, we represent the left-hand limit of $f$ at $a$ by $f(a-):=\lim_{x\uparrow a} f(x)$. Let $\mathbb{I}_{\mathcal{B}}$ denote the indicator function of event $\mathcal{B}$.

All vectors are column vectors throughout the paper unless stated otherwise. However, for the economical use of space, we write $\mathbf{x} = (x_{1}, x_{2}, \ldots, x_{J})$ inside of the text to describe entries of vector $\mathbf{x} \in \R^{J}$. For $j\in\{1,\dots,J\}$, we denote by $\Be_j$ the $J$-dimensional binary vector, in which the $j$-th entry is $1$ and other entires are $0$. The zero vector is denoted by $\bzero$.

With slight abuse of notation, for any function $f: \mathcal{B}\to \R$ with a finite domain $\mathcal{B}$ (i.e., $|\mathcal{B}|<\infty$), we use bold $\bbf$ to denote a $|\mathcal{B}|$-dimensional vector such that $\bbf \defi (f(x), x \in \mathcal{B})$. 
Similarly, for any function $H: \mathcal{B} \times \mathcal{C}\to \R$, where $\mathcal{B}$ and $\mathcal{C}$ are finite, we use  bold $\bH$ to denote the $|\mathcal{B}|\times |\mathcal{C}|$ matrix, the $(b,c)$-entry of which is $f(b,c)$ for $b\in\mathcal{B}$ and $c\in\mathcal{C}$. 

\subsection{The Tandem Queueing System}
\label{sec: the tandem queueing system} 
In this section,  we describe our tandem queueing model denoted by $(B_1,B_2,\dots,B_J)\in\mathbb Z_+^J$ for $J\in\mathbb N$.  
Time is assumed to be continuous and 
an illustration of the system setup is given in Figure~\ref{fig:tandem_system}.

\begin{figure}[ht]
\centering
\includegraphics[width=0.7\textwidth]{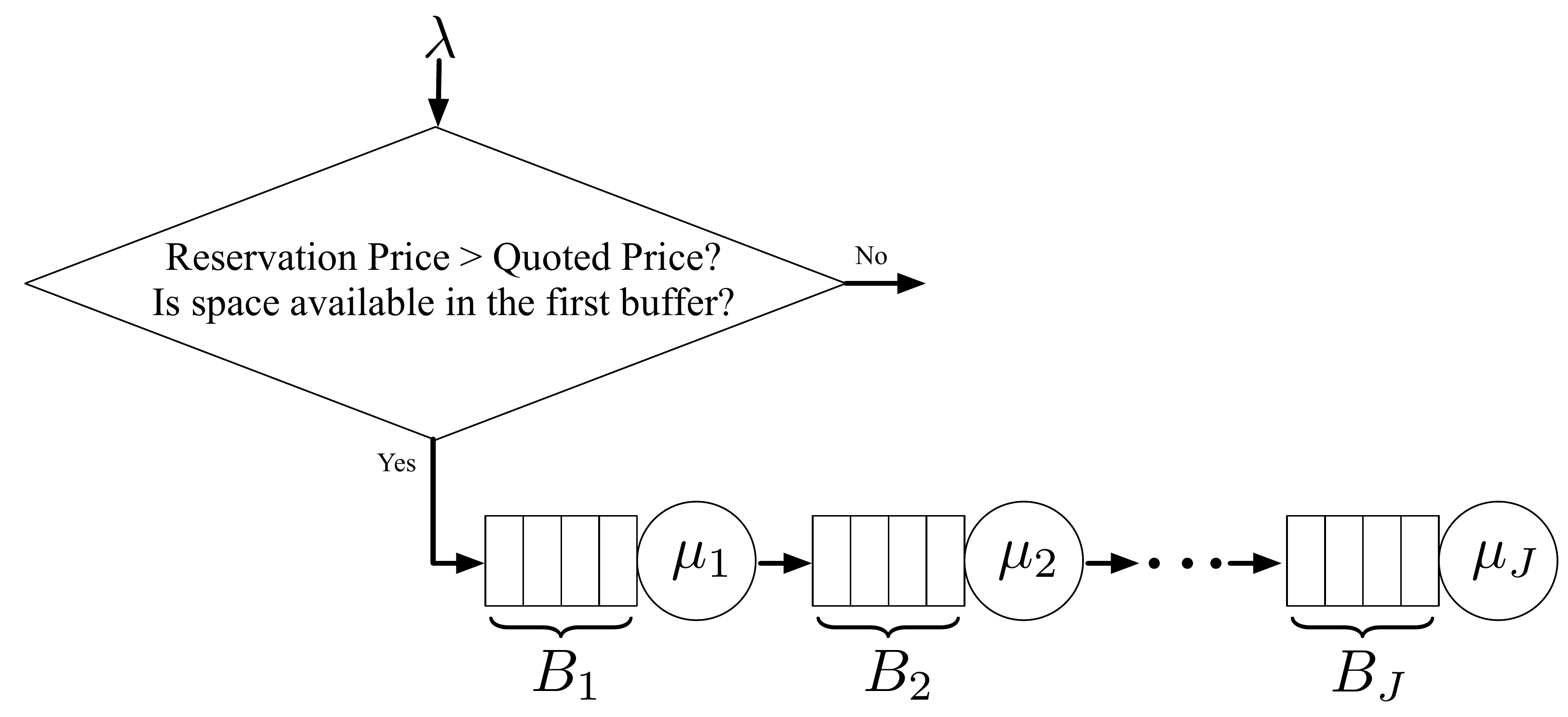}
\caption{Tandem Queueing System $(B_1,B_2,\dots,B_J)$. \vspace{-0.2in}} 
\label{fig:tandem_system}
\end{figure}

\paragraph{System}
The tandem system consists of $J\in\mathbb N$ stations which are lined up one behind another and indexed by $\{1,2,\dots,J\}$. Each station has, before it, a finite buffer  that stores unserviced jobs (customers) and a server that processes at most one job at a time on a first-come, first-served (FCFS) basis. We denote the size of the buffer before station $j$ by $B_{j}\in\mathbb Z_{+}$ so that the maximum number of customers at station $j$ is $(B_{j}+1)$ including the one being served. % at the server's position.
Service times at station $j$ are exponentially distributed with rate $\mu_{j}\in\mathbb{R}^+$ %for $j\in \{1,2,\dots,J\}$ 
and all the service times are independent.

\paragraph{Arrivals} % arrival, price distribution -> virtual arrival and actual arrival?
The first station (station~$1$) receives a stream of incoming customers according to a homogeneous Poisson process with a rate $0 < \lambda < \infty$; that is, inter-arrival times between two adjacent arrivals are independent and exponentially distributed with mean ${1}/{\lambda}$.
Each incoming customer has her own budget for services %, which we call 
called the \emph{reservation price}. 
The service provider quotes a price  from a set $\mathcal{A} \subset \R_{+}$ of prices to customers. 
Then, % following the same assumption made in the literature on pricing for queues (e.g.~\citet{ziya:06,yoon:04,cil:11}), 
we assume that an arriving customer %decides to 
enters the system if 
\begin{enumerate}[itemsep=0in, label=(\arabic*)]
\item 
Station $1$ is not full at the time she arrives;
\item
The quoted price is not greater than her reservation price
\end{enumerate} 
as in literatures on pricing for queues (e.g.~\citet{ziya:06,yoon:04,cil:11}).
The cases with delay-sensitive customers as well as price-sensitive are  out of the main topic of this work and require future research.

Any incoming customer who does not meet either of the above conditions is lost. 
We assume that reservation prices are independent (among customers) and identically distributed random variables
 with cumulative distribution function $F(\cdot)$. 
Thus, when station $1$ is not full and the quoted price is $a\in\mathcal{A}$, an incoming customer enters the system with probability $1-F(a-)$. 

\paragraph{Job Processing and Blocking Rule}  
A customer who enters the system completes a job at every station before leaving the system. 
After the $j$-th job of a customer is processed in station $j$, the customer enters the next station, i.e., station $j+1$, if the buffer  before the next station is not full.  
For the case that the buffer is full, we apply the communication blocking mechanism: Server at station $j$ is not allowed to start service until space is available at station $j+1$. Immediately after all jobs of a customer are processed, the customer leaves the system.

\paragraph{Assumptions} 
Throughout this paper, we assume that the \emph{utilization}, denoted by $\rho$, of the system is less than one:
\begin{align*}
	\rho\ :=\ \lambda\left(\frac{1}{\mu_{1}}+\frac{1}{\mu_{2}}+\dots+\frac{1}{\mu_{J}}\right)\ <\ 1.
\end{align*}
We also assume that, for any quoted price $a\in\mathcal{A}$, the probability that an incoming customer with a bigger reservation price than $a$ is positive: $1-F(a-)>0$.
This assumption is only for technical simplicity and  can be eliminated in some of our results.

\subsection{Policies and Performance Metric}
\label{sec:policies and performance metric}
This section introduces 
state-dependent pricing policies and the performance metric.
% A pricing policy specifies the price that the service provider quotes to the next incoming customer whenever the state of the system changes (i.e., when a customer enters the system or completes service at a station). 
The state of the system is represented by a $J$-dimensional vector $\bs=(s_1,s_2,\dots,s_J) \in\mathbb{Z}_{+}^{J}$, where $s_{j}$ represents the number of jobs waiting in the buffer of station $j$ or being served at the station. Thus, the state space is $\mathcal{S}=[B_1+1]\times[B_2+1]\times\dots\times [B_J+1]$. 
Whenever the state of the system changes (i.e., when a customer enters the system or completes service at a station), 
a state-dependent pricing policy, called a \emph{Markovian deterministic stationary pricing policy} specifies the price that the service provider quotes to the next incoming customer. The set of all Markovian deterministic stationary pricing policies is denoted by $\Pi$. 
Policy $\pi\in\Pi$ can be regarded as function $\pi: \mathcal{S} \to \mathcal{A}$,  where $\pi(\bs)$ is the quoted price under policy $\pi$ when the system resides at state $\bs \in \mathcal{S}$. 
We represent the state of the system under policy $\pi$ at time $t\geq 0$ by $\bX^\pi(t)\in\mathcal S$ (i.e., $X^{\pi}_{j}(t)$ is the number of jobs in station $j$ at time $t$.) Then, $\{\bX^\pi(t)\,:\,t\geq 0\}$ is a continuous-time Markov chain (CTMC) with state space $\mathcal S$ and 
 generator function $Q^\pi:\mathcal{S}\times\mathcal{S}\to \R_+$ such that, for $\bs\neq\bs^\prime$,
\begin{align*}
    Q^\pi(\bs,\bs^\prime)~=~
\arraycolsep=1.4pt\def\arraystretch{0.5}
\left\{
    \begin{array}{ll}
        \lambda (1-F(\pi(\bs)-)) & \textrm{if $\bs^\prime=\bs+\Be_1$}, \\
        \mu_j & \textrm{if $\bs^\prime=\bs-{\Be}_j+{\Be}_{j+1}$ for $j\in\{1,2,\dots,J-1\}$}, \\
        \mu_J & \textrm{if $\bs^\prime=\bs-{\Be}_J$},\\
        0 & \textrm{otherwise},
    \end{array}\right.
\end{align*}
and $Q^\pi(\bs,\bs)=-\sum_{\bs^\prime\neq\bs} Q^\pi(\bs,\bs^\prime)$.

% For pricing policy $\pi\in\Pi$, 
The performance metric that interests us for a pricing policy is, roughly speaking, a long-run expected time-average revenue of the service provider. The total revenue increases by the amount of the quoted price when a customer enters the system (i.e., the reservation price is not less than the quoted price and the buffer before the first station is not full). In other words, under pricing policy $\pi\in\Pi$, whenever the state becomes $\bs+\Be_1$ from $\bs$, the provider gets rewarded by $\pi(\bs)$, so  
the reward function $R^{\pi}:\mathcal{S}\times\mathcal{S}\to\mathcal{A}\cup\{0\}\subset \mathbb R$ associated with $\pi$ is defined by 
\begin{align*}
R^{\pi}(\bs,\bs^\prime)~=~
\left\{
\arraycolsep=1.4pt\def\arraystretch{0.5}
\begin{array}{ll}
\pi(\bs) & \textrm{if $\bs'=\bs+\Be_{1}$}, \\
0 & \textrm{otherwise}.
\end{array}\right.
\end{align*}Then, with reward $\mathbf{R}^{\pi}$, $\{\bX^\pi(t):t\geq 0\}$ is a \emph{continuous-time Markov reward model} (cf.~\cite{Gouberman:14}). For the formal definition of long-run expected time-average  revenue of pricing policy $\pi\in\Pi$, we let $T^{\pi}_{n}$ be the time when $\bX^\pi(t)$ changes: 
\begin{align*}
    % T^{\pi}_{1}~:=~& \inf\{t>0\,:\,\bX^\pi(t)\neq \bX^\pi(t-)\}, \\
    T^{\pi}_{n+1}~:=~ \inf\{t>T^{\pi}_{n}\,:\,\bX^\pi(t)\neq \bX^\pi(t-)\},\quad n=0,1,\cdots,
\end{align*}
where we set $T^{\pi}_{0}=0$ for notational convenience.
Then, the long-run expected time-average revenue of tandem queueing system $(B_1,B_2,\dots,B_J)$ under policy $\pi$, which we call the \emph{gain}  of pricing policy $\pi\in\Pi$, is represented as
\begin{align}
\label{eqn:gain expression 1}
g^{\pi}(B_{1},B_{2},\dots,B_{J})~:=~\lim_{T\to\infty} \frac{1}{T}\; \mathbb E\left[\sum_{n\,:\,T^{\pi}_{n}\leq T} R^\pi(\bX^\pi(T^{\pi}_{n-1}),\bX^\pi(T^{\pi}_{n}))\right],
\end{align}
which is well-defined and identical for any initial state $\bX^{\pi}(0)$ because $\{\bX^{\pi}(t):t\geq 0\}$ is irreducible:
By virtue of the assumptions in Section~\ref{sec: the tandem queueing system}, any state $\bs\in\mathcal {S}$ communicate with $\bzero\in\mathcal{S}$ under any pricing policy $\pi\in\Pi$. Therefore, the Markov reward model under consideration is unichain (or recurrent, more strictly speaking).

\subsection{The MDP Model for Optimal Pricing in Tandem Queueing Systems}
\label{sec:queue rm}
We now state the service provider's optimal pricing problem: the goal is to determine a %deterministic stationary
state-dependent pricing policy that maximizes the gain among all such policies, that is, to solve the continuous-time MDP model, called \emph{true} optimization problem,
\begin{align}
\label{eq:max_gain_problem}
\max_{\pi\in\Pi}\; g^\pi(B_1,B_2,\dots,B_J),
\end{align}
%which we call the \emph{true} optimization problem. 
the optimal objective value of which is referred to as the \emph{optimal gain}. The main obstacles that make true optimization problem \eqref{eq:max_gain_problem} hard to solve include:
\begin{enumerate}[itemsep=0in, label=(\arabic*)]
\item 
The number of feasible solutions (pricing policies) is  $|\mathcal{A}|^{|\mathcal{S}|}$, which grows exponentially with the number of stations and the sizes of buffers; this phenomenon is the so-called ``curse of dimensionality''.    
\item 
Evaluating the gain under a given policy is computationally challenging because the quoted price relies on the underlying state of the system, and in turn, affects the Markov reward process.
\end{enumerate}
To resolve these challenges, for general tandem queueing systems, we focus on shedding light on the performance of static pricing policies, under which the quoted prices are the same for all states at all times, % compared with
rather than dynamic pricing policies that quotes prices dynamically in response to the real-time state of the system. Particularly, for small-sized tandem systems, we also characterize  optimal static and dynamic pricing policies. 

We close this section by introducing another expression of the gain built on the stationary distribution of the system, by which we obtain an upper bound on the gain. Let $\boldsymbol\eta^\pi$ be the stationary distribution of $\{\bX^\pi(t)\,:\,t\geq 0\}$, where $\eta^\pi(\bs)$ is the steady-state probability that the system under pricing policy $\pi\in\Pi$  stays at state $\bs$. The existence and uniqueness of $\boldsymbol\eta^\pi$ follows from the fact that $\{\bX^\pi(t)\,:\,t\geq 0\}$ is a unichain process with a finite state space. 
Then, following from Proposition 4.2 in \cite{Gouberman:14},  we represent the gain under pricing policy $\pi$, defined in~\eqref{eqn:gain expression 1},  as
\begin{align}
\label{eq:gain_pi}
g^{\pi}(B_{1},B_{2},\dots,B_{J})~=~\sum_{\bs\in\mathcal{S}} r^{\pi}_{\cont}(\bs)\eta^{\pi}(\bs),
\end{align}
where
\begin{align*}
r^{\pi}_{\cont}(\bs)~:=~
\left\{
\arraycolsep=1.4pt\def\arraystretch{0.5}
\begin{array}{ll}
\pi(\bs) \lambda(1-F(\pi(\bs)-)) & \textrm{if $s_1\leq B_1$},\\
0 & \textrm{otherwise},
\end{array}\right.\quad\quad\textrm{for $\bs=(s_1,s_2,\ldots,s_{J})\in\mathcal{S}$,}
\end{align*}
is called the \emph{continuized rate reward function}. 
From {\eqref{eq:gain_pi}}, we derive an upper bound on the system's gain the next proposition.% in Proposition~\ref{prop:upper_bound}.  
\begin{proposition}
\label{prop:upper_bound}
For pricing policy $\pi \in \Pi$ for tandem queueing system $(B_1,B_2,\dots,B_J)$, we have
\begin{align*}
g^{\pi}(B_1,B_2,\dots,B_J)~\leq~\max\{a\lambda(1-F(a-)):a\in \mathcal{A}\},
\end{align*}
which is, therefore, an upper bound on the optimal objective value of  true optimization problem~\eqref{eq:max_gain_problem}.
\end{proposition}
\begin{proof}{Proof.}
Let $M\defi\max\{a\lambda(1-F(a-)):a\in\mathcal{A}\}$, which does not depend on a pricing policy.
Since $\pi(\bs)\in \mathcal{A}$, we have
\begin{align*}
r^{\pi}_{\cont}(\bs)~\leq~\lambda\pi(\bs)(1-F(\pi(\bs)-))~\leq~M
\end{align*}
Therefore, from~\eqref{eq:gain_pi}, we obtain that
\begin{align*}
g^\pi(B_1,B_2,\dots,B_J)~\leq ~\sum_{\bs\in\mathcal S} r^{\pi}_{\cont}(\bs)\eta^\pi(\bs)~\leq ~
         M \sum_{\bs\in\mathcal S} \eta^\pi(\bs)~ = ~  M~=~\max\{a\lambda(1-F(a-)):a\in\mathcal{A}\}.
\end{align*}
Since this holds for any policy $\pi\in\Pi$, $\max\{a\lambda(1-F(a-)):a\in\mathcal A\}$ is an upper bound for  $\max_{\pi\in\Pi} g^\pi(B_1,B_2,\dots,B_J)$.
\hfill \halmos
\end{proof}

%%%%%%%%%%%%%%%%%%%%%%%%%%%%%%%%%%%%%%%%%%%%%%%%%%%
%            Summary of main results              %
%%%%%%%%%%%%%%%%%%%%%%%%%%%%%%%%%%%%%%%%%%%%%%%%%%%
\section{Main Results}
\label{sec:summary of main results}
This section states our main results in detail. Section~\ref{sec:static_policies} defines static pricing policies and introduces several related concepts, by which we show the asymptotic optimality of static policies and characterize the gap between the  optimal gain (i.e., the optimal objective value of problem~\eqref{eq:max_gain_problem}) and the gain under the best %optimal
 static pricing policy in Section~\ref{sec:asymptotic}.   
%The remaining results relate to small-sized tandem systems: 
%Section~\ref{sec:conditions} describes sufficient and necessary conditions that guarantee the optimality of static policies and Section~\ref{sec:fast_algorithm} proposes a fast algorithm to find an optimal policy. 
The detailed proofs of these results are provided in Sections~\ref{sec:proof_lower_bound}.

\subsection{Static Policies}
\label{sec:static_policies}
Out of all possible policies (i.e., policies in $\Pi$), we are particularly interested in a class of simply-structured policies, called  \emph{static policies}, which are important enough to warrant a formal definition. 

\begin{definition}[\textbf{Static policy}]
A deterministic stationary policy of tandem queueing systems is \emph{static} if it quotes the same price to every customer regardless of the state of the system.
For static policy $\pi$, we denote by $a^\pi\in\mathcal A$ the price quoted under $\pi$.  On the other hand, for each price $a \in \mathcal{A}$,  $\pi_{a}$ denotes the static policy such that $\pi_{a}(\bs) = a$ for all $\bs\in \mathcal{S}$.  We also let $\Pi_{\static} \defi \{\pi_{a}: a \in \mathcal{A}\}$ be the set of all static policies in $\Pi$.
\end{definition}

The fact that a static policy fixes the quoted price simplifies the system dynamics, particularly, the arrival process. Therefore,
the gain of a static policy has a simple form in terms of the  \emph{potential arrival rate} and the \emph{blocking probability}  associated with the policy, as defined below:
\begin{definition}[\textbf{Potential Arrival Rate} \& \textbf{Blocking Probability}]\ \\
\label{definition:potential arrivals}
    Let $\pi \in \Pi_{\static}$ be a static policy with quoted price $a^{\pi}$. 
    \begin{enumerate}[label=(\arabic*)]
    	\item
The \emph{potential arrival rate}, denoted by $\lambda^{\pi}$, is the arrival rate of customers whose reservation prices are greater than or equal to quoted price $a^{\pi}$:
    	 \begin{align*}
        	\lambda^{\pi}~\defi~\lambda(1-F(a^{\pi}-)).
   		\end{align*}
   		For each price $a\in\mathcal{A}$, we  denote the potential arrival rate of static policy $\pi_{a}$ by $\lambda_{a}\defi \lambda^{\pi_{a}}$.
   		\item
The \emph{blocking probability}, $\beta^\pi$, is the  steady-state probability that the buffer space before station $1$ is fully occupied, namely,
		\begin{align*}
			\beta^{\pi}~:=~\sum_{\bs:s_{1}=B_1+1}\eta^{\pi}(\bs)~=~1-\sum_{\bs:s_1\leq B_1} \eta^{\pi}(\bs),
		\end{align*}
		where we recall that $\boldsymbol{\eta}^{\pi}$ is the stationary distribution of $\{\bX^\pi(t):t\geq 0\}$.
%   		\tdel{For each price $a\in\mathcal{A}$, we also denote the blocking probability of static policy $\pi_{a}$ by $\beta_{a}$.}
  	\end{enumerate}
\end{definition}
\noindent
For static policy $\pi$, the continuized rate reward function $r^{\pi}_{\cont}:\mathcal{S}\to\mathbb R$ is written as
 $r^{\pi}_{\cont}(\bs)\defi a^\pi \lambda(1-F(a^\pi-))=a^{\pi}\lambda^{\pi}$ for $\bs=(s_1,s_2,\dots,s_{J})\in\mathcal{S}$  with $s_1\leq B_1$,
so the gain in \eqref{eq:gain_pi} becomes
\begin{align}
\label{eq:gain_static}
        g^\pi(B_{1},B_{2},\dots,B_{J})~=~\sum_{\bs\in\mathcal{S}} r^{\pi}_{\cont}(\bs){\eta}^{\pi}(\bs)~=~a^\pi \lambda^{\pi} (1-\beta^\pi),
\end{align}
where $\lambda^{\pi}$ is the potential arrival rate and $\beta^{\pi}$ is the blocking probability under policy $\pi$.

\begin{remark}
\label{remark:blocking probability goes to zero}
Note that, if the buffer size $B_{1}$ is infinite, the blocking probability is $\beta^{\pi}=0$ for any (static) policy $\pi$, which implies that 
$g^{\pi}(\infty,B_{2},\dots,B_{J}) = a^{\pi}\lambda (1-F(a^{\pi}-))$ for $\pi \in \Pi_{\static}$.
\end{remark}

\subsection{Asymptotic Optimality of Static Policies}
\label{sec:asymptotic}
This section shows the asymptotic optimality of static policies as 
$B_{1}$ grows, which is established based on Theorem~\ref{thm:asymptotic_static} below:
\begin{theorem}
\label{thm:asymptotic_static}
Let $\pi \in \Pi_{\static}$ be a static policy for tandem queueing system $(B_{1},B_{2},\dots,B_{J})$ and $a^\pi$ be the quoted price under $\pi$.
Then, we have that, for a large enough $B_1$, 
\begin{align*}
g^{\pi}(B_{1},B_{2},\dots,B_{J}) ~\geq~ a^\pi\lambda(1-F(a^\pi-))(1-c\,p^{B_{1}-1})~=~a^\pi\lambda^{\pi}(1-c\,p^{B_{1}-1}),
\end{align*}
where $c$ and $p<1$ are positive constants depend only on $\lambda^{\pi}$ and $\mu_{1}, \mu_{2}, \ldots, \mu_{J}$.
\end{theorem}
% \begin{theorem*}
% Let $\pi \in \Pi_{\static}$ be a static policy for tandem queueing system $(B_{1},B_{2},\dots,B_{J})$ and $a^\pi$ be the quoted price under $\pi$.
% Then,
% there exist positive constants $c$, $p<1$, and $N\in\mathbb N$, all of which depend only on $\mu_{1}, \mu_{2}, \ldots, \mu_{J}$ and $\lambda(1-F(a_{*}-))$,
% such that $B_1\geq N$ implies %\tdel{$\geq 1$}
% \begin{align*}
% g^{\pi}(B_{1},B_{2},\dots,B_{J}) ~\geq~ a^\pi\lambda(1-F(a^\pi-))(1-c\,p^{B_{1}-1})~=~a^\pi\lambda^{\pi}(1-c\,p^{B_{1}-1}).
% \end{align*}
% % for some positive constants $c$ and $p<1$, which depend only on $\lambda^{\pi}$ and $\mu_{1}, \mu_{2}, \ldots, \mu_{J}$.
% \end{theorem*}
\begin{proof} 
\mbox{}\textit{Proof}.
See Section~\ref{sec:proof_lower_bound}.
\end{proof}
\noindent
For any static policy, Theorem~\ref{thm:asymptotic_static} gives a lower bound on the gain under the policy, which is also a lower bound on the objective value for the following \emph{static} optimization problem:
\begin{align}
\label{eq:static_problem}
\max_{\pi\in\Pi_{\static}} g^\pi(B_1,B_2,\dots,B_J).
\end{align}
A static policy that solves~\eqref{eq:static_problem} is called an \emph{optimal static policy}.

Next, we introduce a \emph{simple static policy} in order to specify the gap between the optimal objective value of static optimization problem~\eqref{eq:static_problem} and that of true optimization problem~\eqref{eq:max_gain_problem}. Let $a_{*}\in\mathcal A$ be an optimal solution to the following \emph{upper-bound} optimization problem:
\begin{align}
\label{eq:ideal_gain_problem}
\max_{a\in\mathcal{A}}\;a\lambda(1-F(a-)).
\end{align}
Then, applying Theorem~\ref{thm:asymptotic_static} to the simple static policy, $\pi_{a_{*}}$, which has a quoted price $a_{*}$, yields the asymptotic optimality of static policies.
\begin{corollary}\label{cor:asymptotic}
For tandem queueing system $(B_{1},B_{2},\dots,B_{J})$ with a large enough $B_1$, we have
\begin{align}
\label{eqn:ideal approximates dynamic}
1-c\,p^{B_{1}-1}~\leq~\frac{g^{\pi_{a_*}}(B_{1},B_{2},\dots,B_{J})}{\max_{\pi\in\Pi}g^\pi(B_{1},B_{2},\dots,B_{J})} ~\leq~1,
\end{align}
from which we obtain that
\begin{align}
\label{eqn:static approximates dynamic}
1-c\,p^{B_{1}-1}~\leq~ \frac{\max_{\pi\in\Pi_{\static}}g^{\pi}(B_{1},B_{2},\dots,B_{J})}{\max_{\pi\in\Pi}g^\pi(B_{1},B_{2},\dots,B_{J})}~\leq~1,
\end{align}
where  $c$ and $p<1$ are positive constants depend only on $\mu_{1}, \mu_{2}, \ldots, \mu_{J}$ and $\lambda(1-F(a_{*}-))$.
\end{corollary}
\begin{proof}{Proof.}
Since $\pi_{a_*}\in\Pi_{\static}$, from Theorem~\ref{thm:asymptotic_static}, we have, for a large enough $B_1$,
\begin{align*}
\max_{\pi\in\Pi_{\static}} g^\pi(B_{1},B_{2},\dots,B_{J})~\geq~ g^{\pi_{a_*}}(B_{1},B_{2},\dots,B_{J})~\geq~a_*\lambda(1-F(a_*-))(1-c\,p^{B_1-1}),
\end{align*}
where $c$ and $p<1$ are positive constants that depend only on $\mu_{1}, \mu_{2}, \ldots, \mu_{J}$ and $\lambda^{\pi_{a_*}}=\lambda(1-F(a_{*}-))$.
% where the second inequality comes from Theorem~\ref{thm:asymptotic_static}. 
Since $a_*$ is a solution to problem \eqref{eq:ideal_gain_problem}, by Proposition~\ref{prop:upper_bound}, we have
\begin{align*}
\max_{\pi\in\Pi} g^{\pi}(B_{1},B_{2},\dots,B_{J})~\leq~a_{*}\lambda(1-F(a_{*}-)),
\end{align*}
from which \eqref{eqn:ideal approximates dynamic} and \eqref{eqn:static approximates dynamic} immediately  follow.   
\hfill \halmos
\end{proof}
% \begin{corollary*}
% For tandem queueing system $(B_{1},B_{2},\dots,B_{J})$, there exist positive constants $c$, $p<1$, and $N\in\mathbb N$, all of which depend only on $\mu_{1}, \mu_{2}, \ldots, \mu_{J}$ and $\lambda(1-F(a_{*}-))$,
% such that $B_1\geq N$ implies
% \begin{align}
% \label{eqn:ideal approximates dynamic}
% 1-c\,p^{B_{1}-1}~\leq~\frac{g^{\pi_{a_*}}(B_{1},B_{2},\dots,B_{J})}{\max_{\pi\in\Pi}g^\pi(B_{1},B_{2},\dots,B_{J})} ~\leq~1,
% \end{align}
% so that
% \begin{align}
% \label{eqn:static approximates dynamic}
% 1-c\,p^{B_{1}-1}~\leq~ \frac{\max_{\pi\in\Pi_{\static}}g^{\pi}(B_{1},B_{2},\dots,B_{J})}{\max_{\pi\in\Pi}g^\pi(B_{1},B_{2},\dots,B_{J})}~\leq~1.
% \end{align}
% \end{corollary*}

\textbf{Insights}. 
We learn from Corollary~\ref{cor:asymptotic} that, as $B_{1} \to \infty$ , the maximum gain obtained by solving static optimization problem~\eqref{eq:static_problem} approaches to the optimal objective value of true optimization problem~\eqref{eq:max_gain_problem} exponentially fast. In other words, an optimal static policy gives a gain that well approximates the optimal gain for a general tandem queueing system.

Then, an immediate question follows: \emph{How can we find an optimal static policy, or in other words, how can we solve static optimization problem~\eqref{eq:static_problem}?} Our suggestion is: we  don't need to. From~\eqref{eqn:ideal approximates dynamic}, we conclude that simple static policy $\pi_{a_{*}}$ results in a gain that also converges to the optimal gain exponentially fast as $B_{1}$ tends to infinity. Thus, only solving upper-bound optimization problem~\eqref{eq:ideal_gain_problem}, which is a much easier problem to deal with than optimization problems~\eqref{eq:max_gain_problem} and \eqref{eq:static_problem}, seems convenient and sufficient. 

Note that upper-bound optimization problem \eqref{eq:ideal_gain_problem} only involves the distribution of customers' reservation prices and a feasible set $\mathcal{A}$. If $\mathcal{A}$ is a discrete set, the number of feasible solutions is at most $|\mathcal{A}|$ and solving~\eqref{eq:ideal_gain_problem} can be as easy as a matter of evaluating the objective function value for each element of $\mathcal{A}$. If $\mathcal{A}$ is continuous, the upper-bound optimization problem is a continuous nonlinear optimization problem  and plenty of algorithms are applicable to solve it, such as the trust-region algorithm, gradient-descent method, and interior-point algorithm.  
 Not only is $\pi_{a_*}$ practically tractable to obtain, but also, more importantly, the gain associated with it is close to the optimal gain for a large $B_{1}$, as justified by Corollary~\ref{cor:asymptotic}.

\section{Numerical Experiments}
\label{sec:numerical experiments}
In this section, we use numerical experiments %on tandem queueing systems %under various policies 
to justify our theoretical results obtained in the previous sections and to demonstrate the usefulness and efficiency of the policies and algorithms we proposed.

In particular, Section~\ref{sec:performance of static policies} tests the performance of static policy $\pi_{a^{*}}$, where $a^{*}$ solves upper-bound optimization problem~\eqref{eq:ideal_gain_problem},
and investigates the influence of the buffer size at the first station on the performance.
 In Section~\ref{sec:sensitivity analysis}, we examine  how various parameters, including utilization $\rho$, the number $J$ of stations, and  buffer sizes for stations other than station~1, affect the performance of policy $\pi_{a^{*}}$.
%Lastly, in Section~\ref{sec:fast algorithm versus policy iteration},
%we compare the fast algorithm (i.e., Algorithm~\ref{alg: fast algorithm mixed policy}) and the unichain policy iteration algorithm with regard to their computational times when applied to solve two-station queueing systems.

\subsection{Performance of Static Policies}
\label{sec:performance of static policies}

This section examines asymptotic optimality results of static policy $\pi_{a^{*}}$ (i.e., the impact of the buffer size before the first station on the performance of the policy) and discusses
what buffer sizes can ensure a good performance of the policy in practice.

To this end, we calculate the gain of static policy $\pi_{a^{*}}$, which is referred to as `Simple optimal gain'. Recall again that $a^{*}$ is a solution to the upper-bound optimization problem,
% To access the performance of static policy $\pi_{a^*}$, where $a^{*}$ is a solution to the optimization problem:
\begin{align}
\max \{ a\lambda(1-F(a-)):a\in\mathcal{A}\}. \tag{\ref{eq:ideal_gain_problem} Revisited}
\end{align}
% we calculate the gain of the policy, which we call `Simple optimal gain'
We compare `Simple optimal gain' with `True optimal gain' (i.e., the optimal gain) and `Static optimal gain', which are the optimal objective values of the true optimization problem,
\begin{align}
\max_{\pi\in\Pi} g^\pi(B_1,B_2,\dots,B_J),\tag{\ref{eq:max_gain_problem} Revisited}
\end{align}
and the static optimization problem,
\begin{align}
\max_{\pi\in\Pi_{\static}} g^\pi(B_1,B_2,\dots,B_J),\tag{\ref{eq:static_problem} Revisited}
\end{align}
respectively. 

Note that we evaluate the gain of a static policy such as the `Simple optimal gain' using equation (8.2.1) in \citet{puterman:94}. We apply the unichain policy iteration algorithm \citep[pp.~378]{puterman:94} to get `True optimal gain'.%, the objective value of  true optimization problem~\eqref{eq:max_gain_problem}.

From the comparison of the three gains for various tandem queueing systems, we learn the following:
\begin{enumerate}[label=(\arabic*)]
  \item
The gain of static policy $\pi_{a^{*}}$ is close to True optimal gain even for a moderate-sized buffer in front of station $1$. 
  This observation practically reinforces  our analytical results in Corollary~\ref{cor:asymptotic}, which shows the asymptotic optimality of $\pi_{a^*}$ when $B_{1}$ is large.
% We have established in Corollary~\eqref{cor:asymptotic} that policy $\pi_{a^*}$ produces a gain  close to  True optimal gain as the buffer size before station~1 tends to be large. However, a more appealing behavior shows up in our numerical tests: the gain of static policy $\pi_{a^{*}}$ is close to True optimal gain even for a moderate-sized buffer in front of station $1$.
  \item
The ratio of the gain of policy $\pi_{a^{*}}$ to True optimal gain approaches to $1$ exponentially fast.
  \item
 The use of static policy $\pi_{a^{*}}$ is particularly beneficial in the cases where  solving the true optimization problem is a time-consuming attempt.
This argument is also validated by the numerical results in Section~\ref{sec:sensitivity analysis}.
\end{enumerate}
As representative outputs of our numerical experiments, Figure~\ref{fig:performance} and Table~\ref{tbl:relative ratio} present our results for tandem queueing system $(B,0)$. In the experiments, the arrival rate of customers is $\lambda = 3.6$, the service rates of stations are $\mu_{1}=\mu_{2}=8.0$, the reservation prices are normally distributed with mean $500$ and standard deviation $50$, and the set of all available quoted prices is $\mathcal{A}=\{350,400,450,500, 550, 600,650,700, 750\}$.

\begin{figure}[ht!]
\begin{center}
\arraycolsep=1.4pt\def\arraystretch{0.5}
\captionsetup{justification=centering}
\captionsetup[subfigure]{justification=centering}
\begin{subfigure}[b]{0.49\textwidth}
\includegraphics[width=\textwidth]{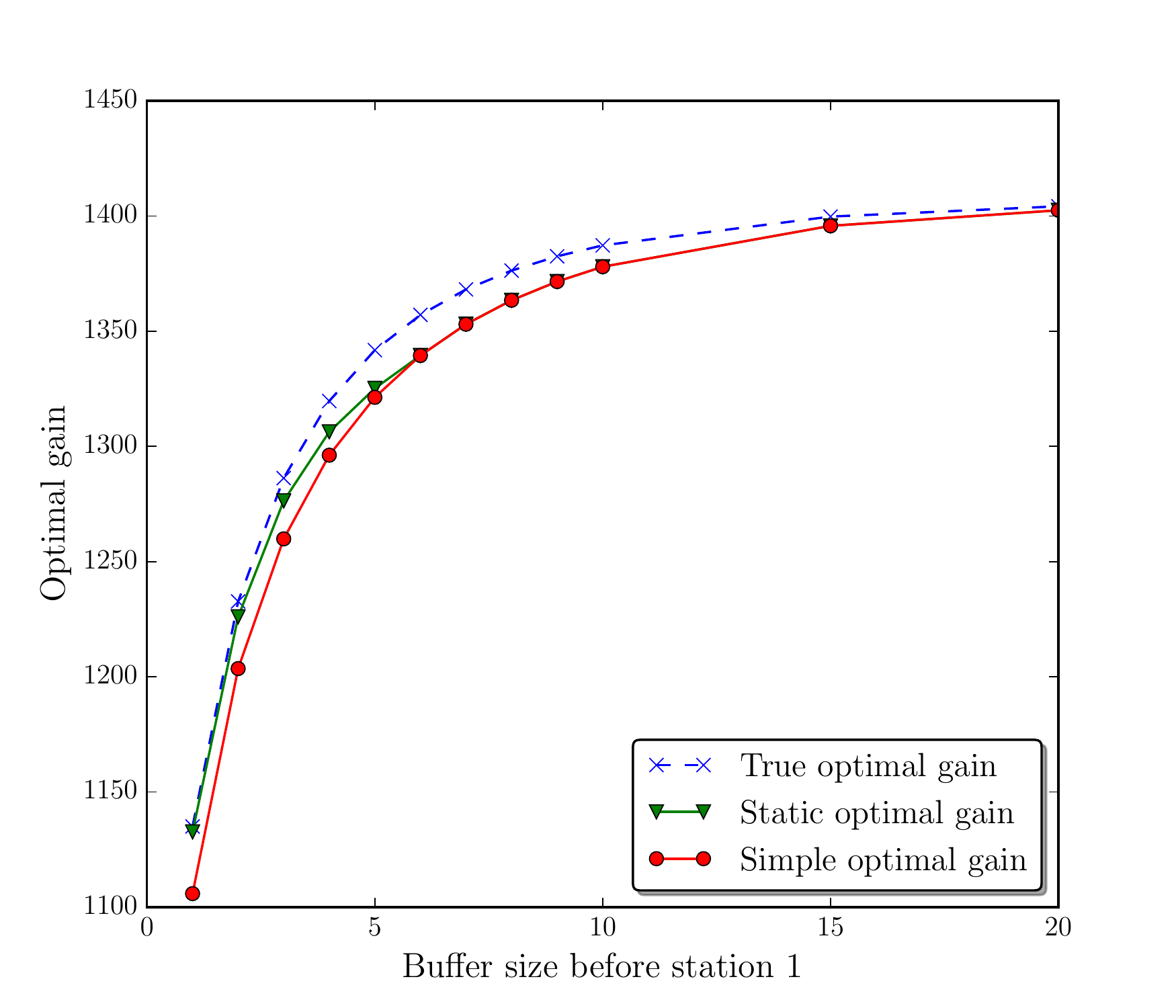}
\caption{Three optimal gains for \\ $B \in \{2,3,\dots,10,15,20\}$}
\label{fig:three optimal gains}
\end{subfigure}
\begin{subfigure}[b]{0.49\textwidth}
\includegraphics[width=\textwidth]{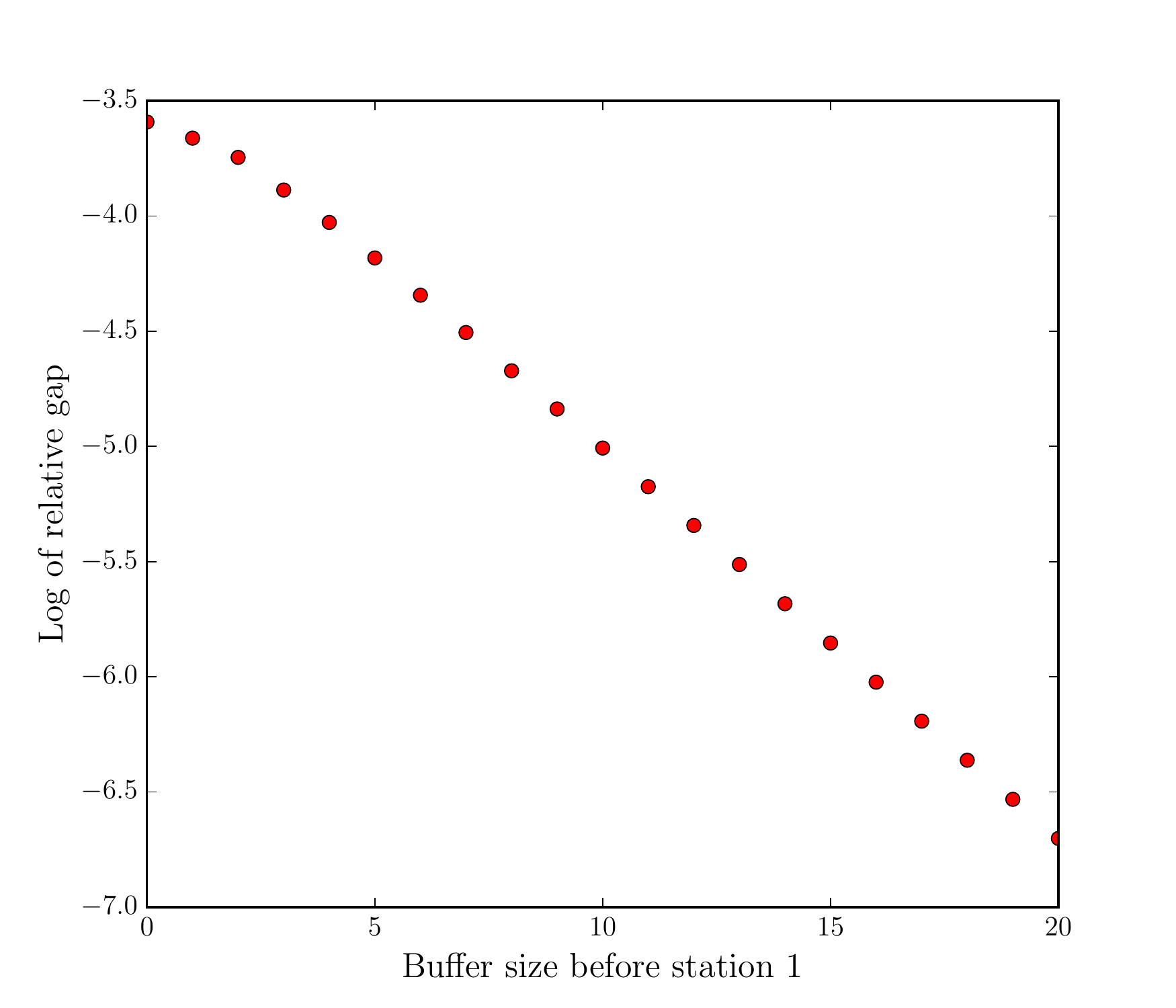}
\caption{Log of relative optimality gap between simple and true optimal gains for $B \in \{2,3,\dots,20\}$}
\label{fig:log of relative gap}
\end{subfigure}
\caption{Performance of $\pi_{a^{*}}$ for tandem queueing system $(B,0)$ depending on $B$:
$\lambda = 3.6$, $\mu_{1}=\mu_{2}=8.0$, reservation prices~$\sim~\text{Normal}(500,50^2)$, and $\mathcal{A}=\{350,400,450,500, 550, 600,650,700, 750\}$.}
\label{fig:performance}
\end{center}
\end{figure}

Figure~\ref{fig:three optimal gains} shows that Simple optimal gain (i.e., the gain of $\pi_{a^{*}}$) becomes closer to True optimal gain as the size of buffer before the first station, $B$,  gets larger as showed in Corollary~\ref{cor:asymptotic}. When $B$ is less than $7$, Simple optimal gain is apparently less than Static optimal gain and True optimal gain, but after $B$ becomes $8$, Simple optimal gain is the same as Static optimal gain. Additionally, when $B$ is greater than $15$, all three gains are very close to each other. The same phenomenon is observed in various settings with different arrival rates, service rates, number of stations, buffer sizes for stations other than station~1, and  distributions of customers' reservation prices. For all simulations we performed, we find that a moderate $B$ (i.e., about $20$) is sufficient to guarantee that the gain of $\pi_{a^{*}}$ is close to  True optimal gain. 

\begin{table}[ht]
\caption{Relative optimality gap between True optimal gain and Simple optimal gain and Computational time for solving true optimal problem~\eqref{eq:max_gain_problem}  using the policy iteration algorithm under various $B$.}
\centering
\arraycolsep=1.4pt\def\arraystretch{0.5}
\begin{tabular}{|r|r|r|}
\hline
\hline
$B$ & Relative optimality gap & Computational time (seconds)\\
\hline
50 & $7.4970\times 10^{-6}$ & 7.2320  \\
100 & $7.1061\times 10^{-9}$ & 12.7700  \\
150 & $<10^{-16}$ & 18.8021 \\
200 & $<10^{-16}$ & 26.5131 \\
250 & $<10^{-16}$& 37.1878 \\
\hline
\hline
\end{tabular}
\label{tbl:relative ratio}
\vspace{-0.2in}
\end{table}

Another notable benefit brought about by  using policy $\pi_{a^{*}}$ relates to the computational time to find $a^{*}$. 
The quoted price, $a^{*}$, is a solution to upper-bound optimization problem \eqref{eq:ideal_gain_problem}, which only involves the distribution of customers' reservation prices and set $\mathcal{A}$. In other words, the same $\pi_{a^{*}}$, once obtained, can be applied to systems with various arrival rates, service rates, and buffer sizes. 
In our simulation setting above, $\mathcal{A}$ is a discrete set and $a^{*}=600$ can be found less than $0.001$ seconds. 
Meanwhile, Table~\ref{tbl:relative ratio} shows the computational times (in seconds) to identify True optimal policy using the policy iteration algorithm along with the relative optimality gap between Simple optimal gain and True optimal gain defined as follows:
\begin{align*}
  \textrm{Relative optimality gap}~=~\frac{\textrm{True optimal gain}-\textrm{Simple optimal gain}}{\textrm{True optimal gain}}.
\end{align*}

Table~\ref{tbl:relative ratio} shows that the computational time increases whereas the relative optimality gap decreases as $B$ becomes large,
which implies that using static policy $\pi_{a^{*}}$ is a better choice when solving true optimization problem~\eqref{eq:max_gain_problem} is hard. Through numerical experiments, we find that, for three-station systems, solving true optimization problem~\eqref{eq:max_gain_problem}  using the policy iteration algorithm is nearly a  mission impossible because of the lack of memory and the extremely long computational time it requires, even when the buffers before all the three stations are moderate-sized.

We also illustrate the rate of convergence of Relative optimality gap in Figure~\ref{fig:log of relative gap}, which is at least exponentially fast  as  $B$ grows according to Corollary~\ref{cor:asymptotic}. 
For a better representation, we plot the log-scale of Relative optimality gap versus $B$  in Figure~\ref{fig:log of relative gap}.
Since Log of Relative optimality gap decreases linearly as $B$ increases, we conclude that the gain of policy $\pi_{a^*}$ converges to `True optimal gain' exponentially fast as $B$ becomes large.  As before, in all tandem queueing systems we tested, the convergence rate appears exponential, justifying a tight bound on the scale of the convergence rate  in Corollary~\ref{cor:asymptotic}.
  
\subsection{Sensitivity Analysis}
\label{sec:sensitivity analysis}
In this section, we focus on the influence of various parameters  on the performance of static policy $\pi_{a^{*}}$. Specifically, we are interested in how the changes of utilization $\rho$ and the number $J$ of stations affect the optimality gap between Simple and True optimal gains. 

First of all, Figures~\ref{fig:gains_2_9}--\ref{fig:gains_3_99} show Simple, Static, and True optimal gains for three different tandem queueing systems: (i) The first system in Figure~\ref{fig:gains_2_9} is exactly the same as that in Section~\ref{sec:performance of static policies}. (ii) The second system in Figure~\ref{fig:gains_2_99} is the same as the first system except for possessing a larger utilization than the first system. (iii) The third one in Figure~\ref{fig:gains_3_99} is the same as the second system except for having one more station than the second system. 
\vspace{-0.2in}
\begin{figure}[ht!]
\centering
%\captionsetup{justification=centering}
\includegraphics[width=0.65\textwidth]{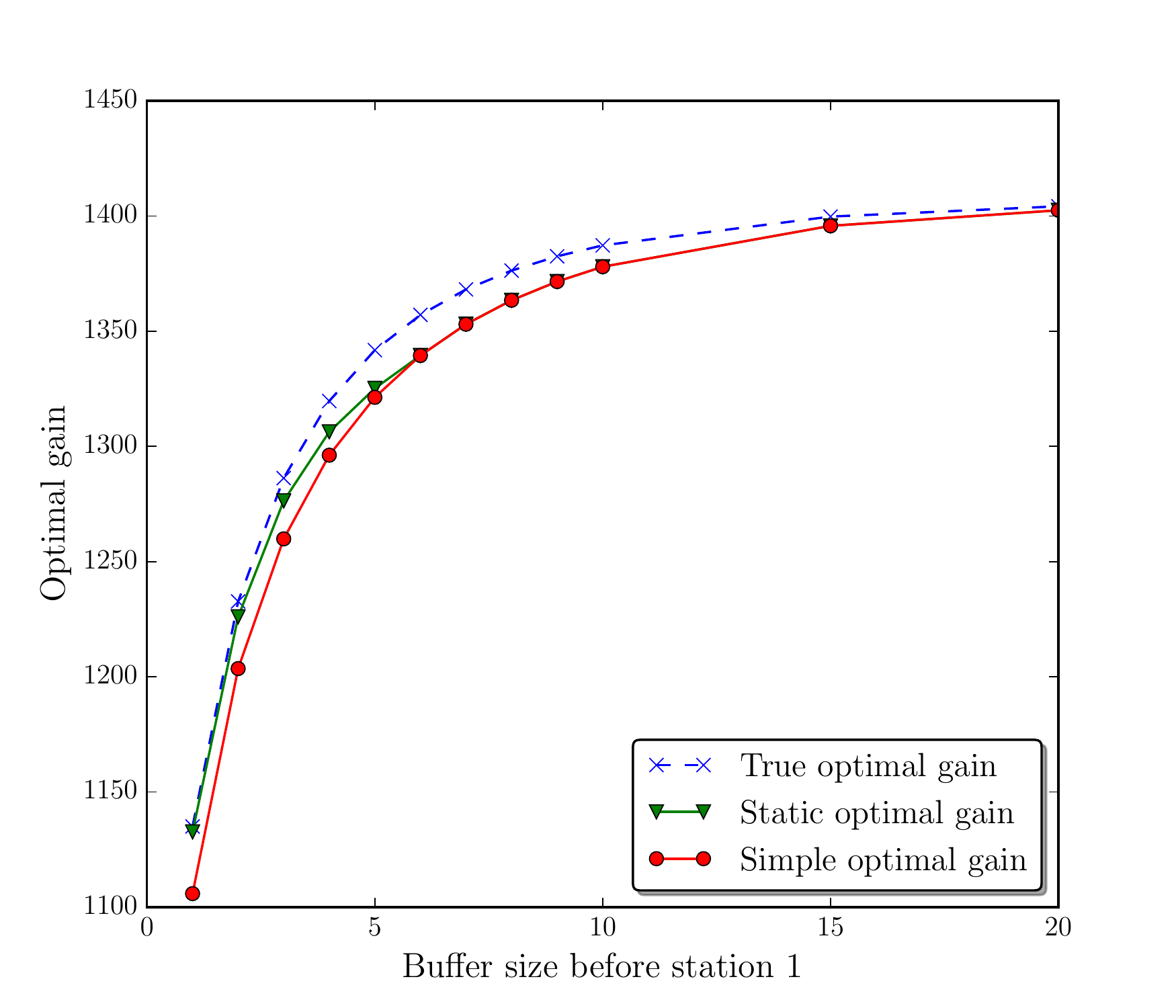}
\caption{True, static, and simple optimal gains for various tandem queueing systems with $J=2$ and $\rho=0.9$. \vspace{-0.2in}}
\label{fig:gains_2_9}
\end{figure}

\begin{figure}[ht]
	\begin{center}
	\captionsetup{justification=centering}
	\captionsetup[subfigure]{justification=centering}	
		\includegraphics[width=0.65\textwidth]{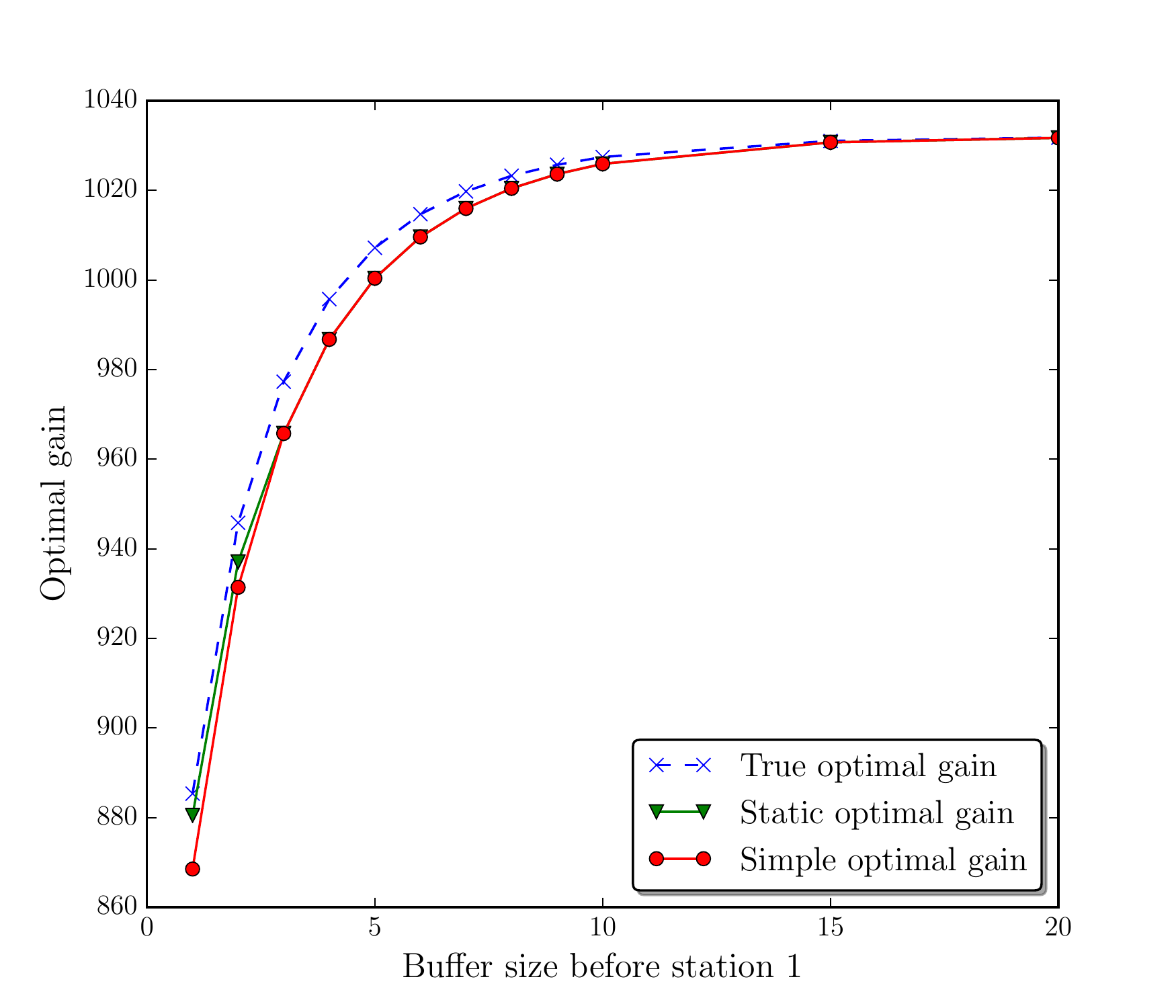}
	\caption{True, static, and simple optimal gains for various tandem queueing systems with $J=2$ and $\rho=0.99$. \vspace{-0.2in}}
	\label{fig:gains_2_99}
	\end{center}
\end{figure}

\begin{figure}[ht]
	\begin{center}
	\captionsetup{justification=centering}
	\captionsetup[subfigure]{justification=centering}
	\includegraphics[width=0.65\textwidth]{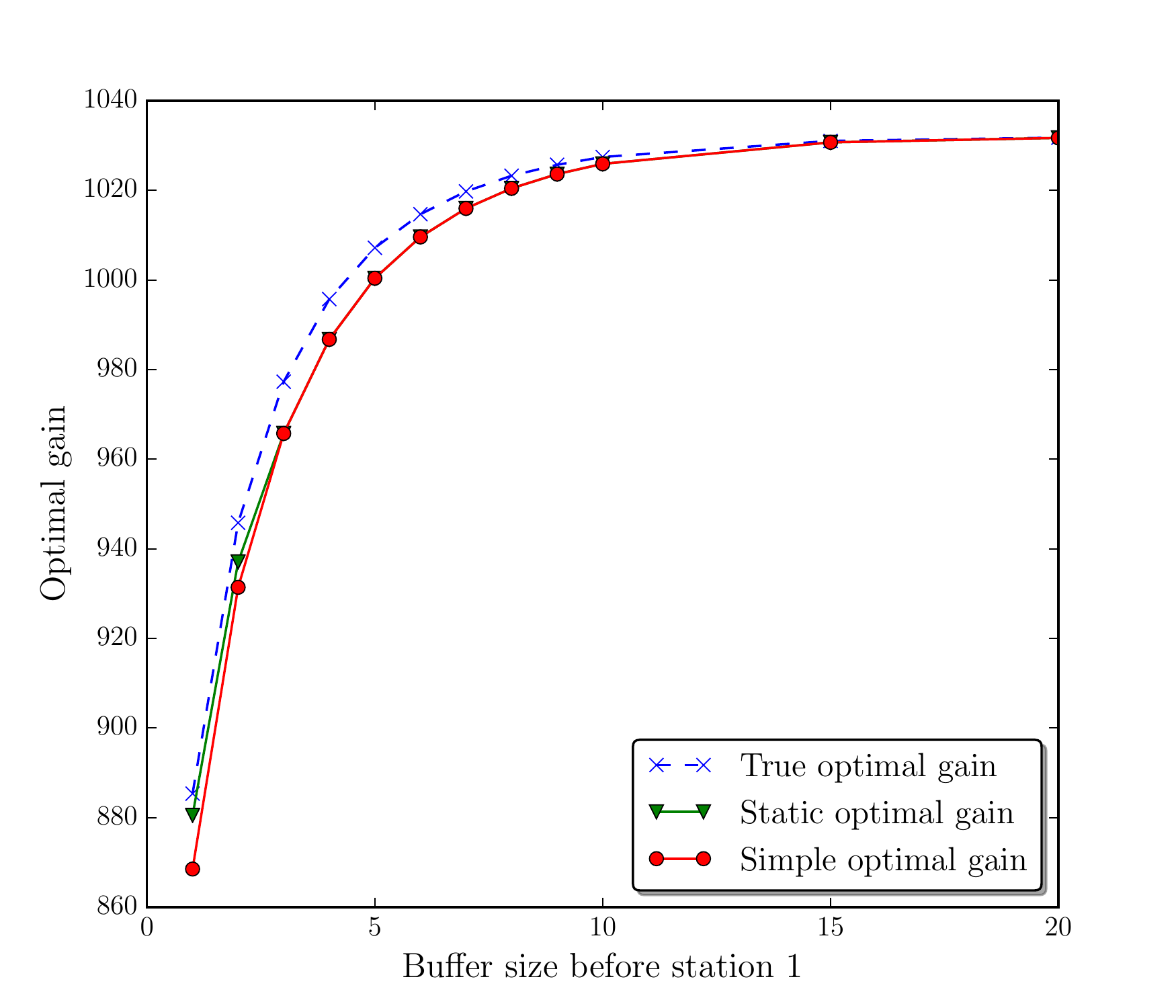}
	\caption{True, static, and simple optimal gains for various tandem queueing systems with $J=3$ and $\rho=0.99$. \vspace{-0.0in}}
	\label{fig:gains_3_99}
	\end{center}
\end{figure}

We find that the performance of $\pi_{a^*}$ improves as the utilization, $\rho$, decreases or the number of stations, $J$, increases. 

We further validate this argument by measuring the optimality gap under various utilizations (Figure~\ref{fig:utilization}) and various numbers of stations (Figure~\ref{fig:size}). Intuitively, these observations make perfect sense since either the decrease in $\rho$ or increase in $J$ reduces the blocking probability, which results in a large gain under static policy $\pi_{a^*}$ according to \eqref{eq:gain_static}.

\begin{figure}[ht!]
\begin{center}
\captionsetup{justification=centering}
\captionsetup[subfigure]{justification=centering}
\begin{subfigure}[b]{0.495\textwidth}
\includegraphics[width=\textwidth]{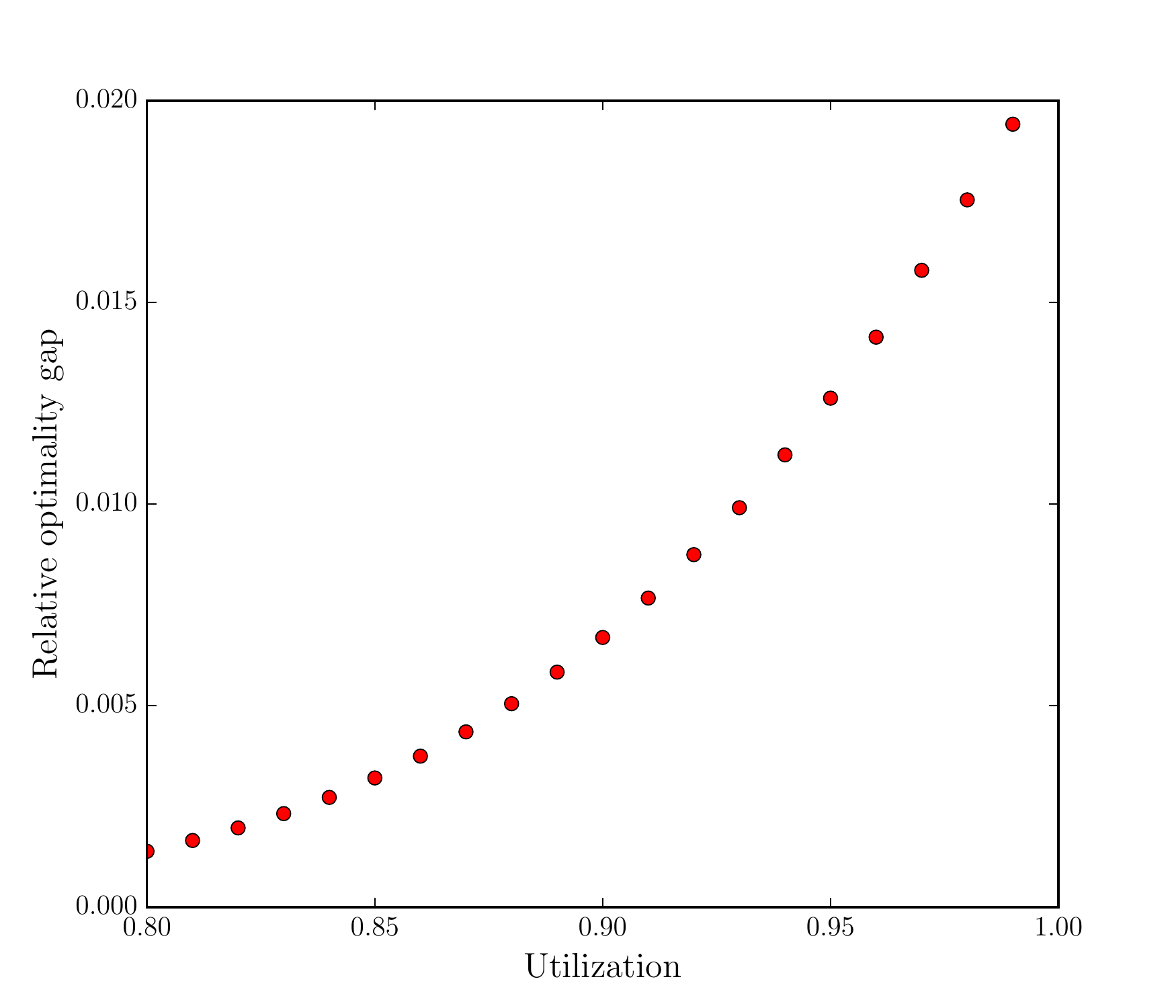}
\caption{Relative optimality gap v.s.~$\rho$}
\label{fig:utilization}
\end{subfigure}
\begin{subfigure}[b]{0.495\textwidth}
\includegraphics[width=\textwidth]{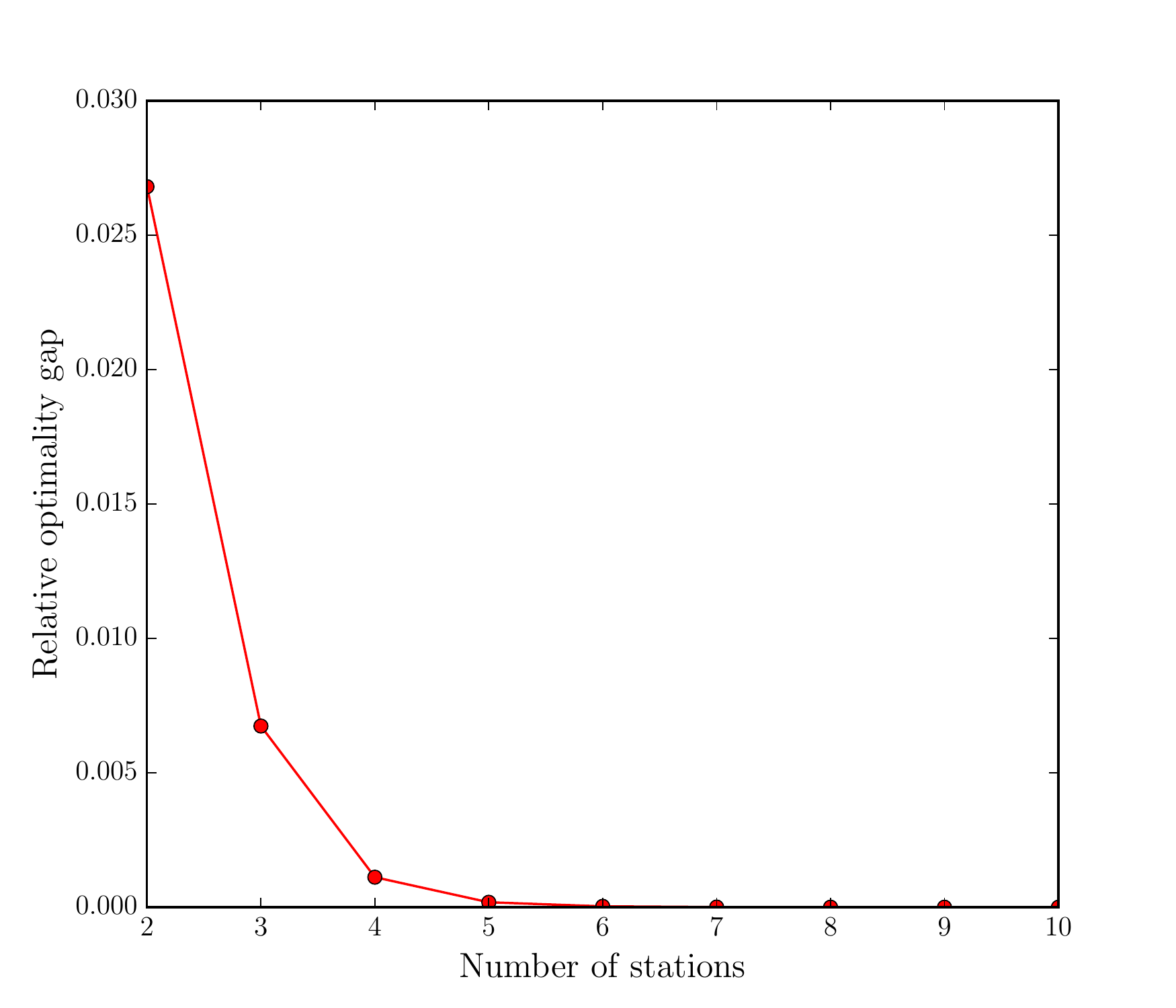}
\caption{Relative optimality gap v.s. $J$}
\label{fig:size}
\end{subfigure}
\caption{Relative optimality gap  of $\pi_{a^{*}}$ for the tandem queueing system: $\lambda = 8.0$, $B_{1}=5$, $B_{j}=0$ for $j \geq 3$, reservation prices follow Normal$(500,50^2)$,  and $\mathcal{A}=\{350,400,450,500, 550, 600,650,700, 750\}$.\vspace{-0.0in}}
\end{center}
\end{figure}
\section{Proof of Asymptotic Optimality of Static Policies}
\label{sec:proof_lower_bound}

In this section, we prove Theorem~\ref{thm:asymptotic_static}, which  provides
a lower bound for $g^{\pi}(B_1,B_2,\dots,B_J)$, the gain of tandem queueing system $(B_1,B_2,\dots,B_J)$ under static policy $\pi$. The key steps in the proof include introducing a  new queueing system coupled with tandem queueing system $(B_1,B_2,\dots,B_J)$ and showing that the coupled system has a gain less than or equal to $g^{\pi}(B_1,B_2,\dots,B_J)$.
Throughout this section, $\pi$ is a fixed static policy with quoted price $a^{\pi}$ and $\lambda^\pi=\lambda(1-F(a^\pi-))$ is the potential arrival rate.

We first define important notions related to arrivals and derive a new expression for gain of $\pi$.
\begin{enumerate}
        \item[(1)] {\bf Potential Arrivals.} 
                We call a customer whose reservation price is not less than the quoted price a \emph{potential customer (arrival)}.
                Since, under static policy $\pi$, the service provider quotes price $a^\pi$ for all states, any incoming customer is a potential customer with probability $1-F(a^\pi-)$.
                Thus, the arrival process of potential customer is Poisson with rate $\lambda^{\pi}$.
        \item[(2)] {\bf Actual Arrivals.} 
                Potential customers who arrive at non-saturated station $1$  and join the system are called \emph{actual customers (arrivals)}.   
                For $t\geq 0$, we denote the cumulative number of all actual arrivals up to time $t$ by $A(t)$. For notational simplicity, we omit $\pi$ for process $A(t)$ even though $A(t)$ depends on $\pi$.  
\end{enumerate}
Since the total revenue up to time $t$ is $a^{\pi}A(t)$, the gain for tandem queueing system $(B_{1},B_{2},\dots,B_{J})$ under policy $\pi$ becomes
\begin{align}\label{gain_from_A}
    g^\pi(B_1,B_2,\dots,B_J)~=~\lim_{t\to\infty} \frac{1}{t}\;\E\left[ a^\pi A(t) \right]
    ~=~\lim_{t\to\infty} a^\pi\; \E\left[ \frac{A(t)}{t} \right],
\end{align}
which is equivalent to~\eqref{eqn:gain expression 1} and well-defined. Since the cumulative number of actual arrivals is always less than or equal to that of potential arrivals, which is Poisson with rate $\lambda^{\pi}$, we have 
\begin{align} \label{eq:upperbound of A(t)/t}
	\lim_{t\to\infty} \E\left[\frac{A(t)}{t}\right]~\leq~\lambda^{\pi},
\end{align}
from which we get $g^{\pi}(B_1,B_2,\dots,B_J)\leq a^{\pi}\lambda^{\pi}$.

\paragraph{$M/\hypo/1/B_1$-queueing system}
Now, we introduce a new queueing system, which consists of a finite buffer and one server such that:
\begin{enumerate}[label=(\arabic*), itemsep=0in]
	\item
The arrival process is Poisson with rate $\lambda^\pi$. In other words, the arrival process is the same as the potential arrival process in the tandem queueing system $(B_{1},B_{2},\dots,B_{J})$ under policy $\pi$;
	\item
The service time of a customer is the sum of $J$ independent exponential (i.e., hypo-exponential) random variables and  the $j$-th exponential random variable has mean $1/\mu_j$ for $j \in \{1,2,\ldots,J\}$;
	\item
The size of the buffer is $B_1$.
\end{enumerate}
We refer to the above-defined system as the $M/\hypo/1/B_1$-queueing system. In this system, any customer arriving at the system with a saturated buffer is lost. We call a customer who is not lost an \emph{actual customer} and represent the cmulative number of all actual customers up to time $t$ by $\tilde{A}(t)$. 

In the remainder of this section, we introduce two main features of the $M/\hypo/1/B_1$-queueing system in Propositions~\ref{prop:lower bound long-run average actual arrivals} and \ref{prop:A_geq_tilde_A} and prove Theorem~\ref{thm:asymptotic_static} from those features. 

One of important reasons for proposing the $M/\hypo/1/B_1$-queueing system is that the long-run expected time-average number of actual customers for this system has a lower bound, which converges to $\lambda^{\pi}$, an upper bound of $\lim_{t\to\infty} \E[A(t)/t]$ as in~\eqref{eq:upperbound of A(t)/t}, exponentially fast as buffer size $B_{1}$ grows: %, as stated in Proposition~\ref{prop:lower bound long-run average actual arrivals}. 
\begin{proposition}
\label{prop:lower bound long-run average actual arrivals}
    For the $M/\hypo/1/B_1$-queueing system with a large enough $B_1$, we have
    \begin{align*}
        \lim_{t\to\infty} \E\left[\frac{\widetilde{A}(t)}{t}\right]%~=~\lambda^\pi(1-\widetilde\beta)
        ~\geq~ \lambda^{\pi}(1-c\,{p}^{B_1-1}),
    \end{align*}
    where $c$ and $p<1$ are positive constants that depend only on $\mu_{1},\mu_{2},\dots,\mu_{J}$, and $\lambda^{\pi}$.
\end{proposition}
% \begin{lemma*}
%   There exist positive constants $c$, $p<1$, and $N\in\mathbb N$, all of which depend only on $\mu_{1}, \mu_{2}, \ldots, \mu_{J}$ and $\lambda^{\pi}$,
%   such that for $M/\hypo/1/B_1$ queueing system with $B_1\geq N$, we have
%     \begin{align*}
%         \lim_{t\to\infty} \E\left[\frac{\widetilde{A}(t)}{t}\right]%~=~\lambda^\pi(1-\widetilde\beta)
%         ~\geq~ \lambda^{\pi}(1-c\,{p}^{B_1-1}).
%     \end{align*}
% \end{lemma*}
\proof{Sketch of Proof.} 
We prove the result by showing that 
\begin{enumerate}
	\item[(1)] $\widetilde{\beta}\leq c\,{p}^{B_{1}-1}$ for large enough $B_{1}$, where $c$ and $p<1$ are positive constants that depend only on $\mu_{1},\mu_{2},\dots,\mu_{J}$ and $\tilde{\lambda}$,
	\item[(2)] $\lim_{t\to\infty} \E\left[ \widetilde{A}(t)/t \right]~=~\lambda^{\pi}(1-\widetilde{\beta})$,
\end{enumerate}
where $\tilde{\beta}$ is the \emph{blocking probability}, the steady-state probability that an incoming customer is lost. \\
To verify the first statement, we express the blocking probability in terms of the queue length distribution in the $M/\hypo/1/\infty$-queueing system, which decays exponentially according to classical results in \cite{neuts:81}. The second statement follows from the Poisson Arrivals See Time Average (PASTA) Theorem and Vitali's  Convergence Theorem. Section~\ref{app:blocking_probability} details the proof.
\hfill \halmos
\endproof

The other feature is that we can \emph{couple} the two systems so that the cumulative number of actual customers up to time $t$ in the $M/\hypo/1/B_{1}$-queueing system is less than or equal to that in the tandem queueing system $(B_1,B_2,\dots,B_J)$:
\begin{proposition}\label{prop:A_geq_tilde_A}
There exists a common probability space, on which we have, almost surely,
\begin{align*}
    A(t)~\geq~\widetilde{A}(t)\quad\textrm{for all $t\geq 0$},
\end{align*}
where $A(t)$ and $\tilde{A}(t)$ are the cumulative numbers of actual customers up to time $t$ in the tandem queueing system $(B_1,B_2,\dots,B_J)$  and the $M/\hypo/1/B_1$-queueing system, respectively.
\end{proposition}
\proof{Sketch of Proof.}
	The key idea for the proof is that, if we replace the communication blocking mechanism for the tandem queueing system $(B_{1},B_{2},\dots,B_{J})$ by the following rule:
	\begin{center}
		\emph{The server at station $1$ is not allowed to start service until all other stations are empty},
	\end{center}
	the system becomes the $M/\hypo/1/B_{1}$-queueing system. 
	Intuitively, the arrival time for the $n$th actual customer in the latter system is not greater than that in the former system, which implies the above inequality.
	% Intuitively, when the first buffer is saturated, the new rule requires more time before the next actual customer enters the system than the communication blocking mechanism. 
	Section~\ref{sec:proof of coupling} details the proof.
\hfill \halmos
\endproof
% The relation between the two systems can be understood better with the following argument. 

% Proposition~\ref{prop:A_geq_tilde_A} compares $A(t)$ and $\tilde{A}(t)$ based on the coupling of the two queueing systems.
Finally, we formally prove  Theorem~\ref{thm:asymptotic_static} from Propositions~\ref{prop:lower bound long-run average actual arrivals} and~\ref{prop:A_geq_tilde_A}:\\
Let $c$ and $p$ be constant in Proposition~\ref{prop:lower bound long-run average actual arrivals}.
Proposition~\ref{prop:A_geq_tilde_A} implies $\E[A(t)/t]\geq\E[\widetilde{A}(t)/t]$ for $t>0$, so
\begin{align*}
    g^\pi(B_1,B_2,\dots,B_J)~=~
    a^\pi\lim_{t\to\infty}\E\left[\frac{A(t)}{t}\right]
    ~\geq~a^\pi\lim_{t\to\infty}\E\left[ \frac{\widetilde{A}(t)}{t}\right] 
    ~\geq~ a^\pi\lambda^{\pi}\left(1-c\,p^{B_1-1}\right),
\end{align*}
where the first equation follows from \eqref{gain_from_A},
which  proves Theorem~\ref{thm:asymptotic_static}. % \hfill \halmos

\section{Concluding Comments}
\label{sec:concluding comments}

Tandem queueing systems are widely-used stochastic models arising from real-life operations systems such as the processing systems at container terminals. 
This work studies the optimal pricing problem that is formulated as an MDP model aimed at maximizing the long-run expected time-average revenue of a service provider who manages a tandem queueing system with finite buffers.

We study a general tandem queueing system with arbitrarily many stations. We show that the optimal static pricing policy results in a gain that converges to the optimal gain at an exponential rate as the buffer size at the first station becomes large. This result suggests a viable solution to service providers: They can solve the easier static dynamic pricing problem to obtain a well-performed pricing policy rather than dealing with the much harder dynamic pricing problem. More interestingly, we propose a \emph{simple static pricing policy} obtained from  solving an easy  optimization problem (i.e., upper-bound optimization problem~\eqref{eq:ideal_gain_problem}) and show that the simple policy is asymptotically optimal as the buffer size before the first station becomes large, which provides an easy-to-get and nearly optimal policy to practitioners. 

%Next, we fully investigate a special two-station tandem queueing system, in which the service provider quotes from a high and low price and station~2 has a buffer size of zero. For the special system, we characterize the sufficient and necessary conditions that guarantee the optimality of static pricing policies. A fast algorithm is also developed to determine an optimal dynamic pricing policy based on the monotone structure of the optimal dynamic pricing policies for the system. 

A number of numerical experiments are performed to justify our theoretical results and to test the impact of various system parameters on the performance of the simple static policy. 
Simulation results yield plenty of insightful observations, which includes that (i) the simple static pricing policy performs quite well even when the buffer before the first station is moderate-sized and (ii) the gap between the gain under the simple static pricing policy and the optimal gain vanishes with an exponential rate as the buffer size before the first station goes to infinity. This work may be extended toward several directions. For instance, one could investigate the performance of optimal static pricing policies for a tandem line with non-Markovian arrival and service processes. One may also study optimal pricing for a network of queues with a general topology.

\bibliographystyle{informs2014} % outcomment this and next line in Case 1
\bibliography{MDPreference}{} % if more than one, comma separated

\begin{thebibliography}{35}
\providecommand{\natexlab}[1]{#1}
\providecommand{\url}[1]{\texttt{#1}}
\providecommand{\urlprefix}{URL }

\bibitem[{Af\`{e}che \protect\BIBand{} Ata(2013)}]{afeche:13}
Af\`{e}che P, Ata B (2013) Bayesian dynamic pricing in queueing systems with
  unknown delay cost characteristics. \emph{Manufacturing \& Service Operations
  Management} 15(2):292--304.

\bibitem[{Af\`{e}che \protect\BIBand{} Pavlin(2016)}]{afeche:16}
Af\`{e}che P, Pavlin JM (2016) Optimal price/lead-time menus for queues with
  customer choice: {Segmentation}, pooling, and strategic delay.
  \emph{Management Science} 62(8):2412--2436.

\bibitem[{{Aktaran-Kalayci} \protect\BIBand{} Ayhan(2009)}]{aktaran:09}
{Aktaran-Kalayci} T, Ayhan H (2009) Sensitivity of optimal prices to system
  parameters in a steady-state service facility. \emph{European Journal of
  Operational Research} 193(1):120--128.

\bibitem[{Altiok(2000)}]{altiok:00}
Altiok T (2000) Tandem queues in bulk port operations. \emph{Annals of
  Operations Reserach} 93(1):1--14.

\bibitem[{Banerjee \protect\BIBand{} Gupta(2012)}]{banerjee:12}
Banerjee A, Gupta UC (2012) Reducing congestion in bulk-service finite-buffer
  queueing system using batch-size-dependent service. \emph{Performance
  Evaluation} 69(1):53--70.

\bibitem[{Bar-Lev et~al.(2013)Bar-Lev, Blanc, Boxma, Janssen, \protect\BIBand{}
  Perry}]{Bar-Lev:13}
Bar-Lev SK, Blanc H, Boxma O, Janssen G, Perry D (2013) Tandem queues with
  impatient customers for blood screening procedures. \emph{Methodology and
  Computing in Applied Probability} 15(2):423--451.

\bibitem[{\c{C}il et~al.(2011)\c{C}il, Karaesmen, \protect\BIBand{}
  {\"O}rmeci}]{cil:11}
\c{C}il EB, Karaesmen F, {\"O}rmeci EL (2011) Dynamic pricing and scheduling in
  a multi-class single-server queueing system. \emph{Queueing Systems: Theory
  and Applications} 67(4):305--331.

\bibitem[{\c{C}il et~al.(2009)\c{C}il, {\"O}rmeci, \protect\BIBand{}
  Karaesmen}]{cil:09}
\c{C}il EB, {\"O}rmeci EL, Karaesmen F (2009) Effects of system parameters on
  the optimal policy structure in a class of queueing control problems.
  \emph{Queueing Systems: Theory and Applications} 61(4):273--304.

\bibitem[{Cheng \protect\BIBand{} Yao(1993)}]{cheng:93}
Cheng DW, Yao DD (1993) Tandem queues with general blocking: {A} unified model
  and comparison results. \emph{Discrete Event Dynamic Systems: {Theory} and
  Applications} 2:207--234.

\bibitem[{Ching et~al.(2009)Ching, Choi, Li, \protect\BIBand{}
  Leung}]{ching:00}
Ching W, Choi S, Li T, Leung IKC (2009) A tandem queueing system with
  applications to pricing strategy. \emph{Journal of Industrial and Management
  Optimization} 5(1):103--114.

\bibitem[{Cooper(1972)}]{cooper:72}
Cooper R (1972) \emph{Introduction to queueing theory} (Macmillan).

\bibitem[{Dijk \protect\BIBand{} Lamond(1988)}]{Dijk:88}
Dijk NMV, Lamond BF (1988) Simple bounds for finite single-server exponential
  tandem queues. \emph{Operations Research} 36(3):470--477.

\bibitem[{{EUROGATE}(2014)}]{eurogate:14}
{EUROGATE} (2014) Prices and conditions for {EUROGATE} container terminal
  {Hamburg GmbH}. Technical report.

\bibitem[{Folland(1999)}]{Folland:99}
Folland GB (1999) \emph{Real analysis}. Pure and Applied Mathematics (New York)
  (John Wiley \& Sons, Inc., New York), second edition, modern techniques and
  their applications, A Wiley-Interscience Publication.

\bibitem[{Ghoneim \protect\BIBand{} Stidham(1985)}]{ghoneim:85}
Ghoneim HA, Stidham S (1985) Control of arrivals to two queues in series.
  \emph{European Journal of Operational Research} 21(3):399--409.

\bibitem[{Gouberman \protect\BIBand{} Siegle(2014)}]{Gouberman:14}
Gouberman A, Siegle M (2014) \emph{Markov Reward Models and Markov Decision
  Processes in Discrete and Continuous Time: {Performance} Evaluation and
  Optimization}, 156--241 (Berlin, Heidelberg: Springer Berlin Heidelberg).

\bibitem[{Hassin \protect\BIBand{} Koshman(2017)}]{hassin:17}
Hassin R, Koshman A (2017) Profit maximization in the {M/M/1} queue.
  \emph{Operations Research Letters} 45(5):436--441.

\bibitem[{Haviv \protect\BIBand{} Randhawa(2014)}]{haviv:14}
Haviv M, Randhawa RS (2014) Pricing in queues without demand information.
  \emph{Manufacturing \& Service Operations Management} 16(3):401--411.

\bibitem[{Le \protect\BIBand{} Hossain(2008)}]{Le:08}
Le L, Hossain E (2008) Tandem queue models with applications to qos routing in
  multihop wireless networks. \emph{IEEE Transactions on Mobile Computing}
  7(8):1025--1040.

\bibitem[{Low(1974{\natexlab{a}})}]{low:74a}
Low DW (1974{\natexlab{a}}) Optimal dynamic pricing policies for an {$M/M/s$}
  queue. \emph{Operations Research} 22(3):545--561.

\bibitem[{Low(1974{\natexlab{b}})}]{low:74b}
Low DW (1974{\natexlab{b}}) Optimal pricing for an unbounded queue. \emph{IBM
  Journal of Research and Development} 18(4):290--302.

\bibitem[{Maglaras(2006)}]{maglaras:06}
Maglaras C (2006) Revenue management for a multiclass single-server queue via a
  fluid model analysis. \emph{Operations Research} 54(5):914--932.

\bibitem[{Maoui et~al.(2009)Maoui, Ayhan, \protect\BIBand{} Foley}]{maoui:09}
Maoui I, Ayhan H, Foley RD (2009) Optimal static pricing for a service facility
  with holding cost. \emph{European Journal of Operational Research}
  197(3):912--923.

\bibitem[{Marin \protect\BIBand{} Bul\`{o}(2011)}]{martin:11}
Marin A, Bul\`{o} SR (2011) Explicit solutions for queues with hypo-exponential
  service time and applications to product-form analysis. \emph{Proceedings of
  the 5th International ICST Conference on Performance Evaluation Methodologies
  and Tools}, 166--175, VALUETOOLS '11, ISBN 978-1-936968-09-1.

\bibitem[{Mickens(2015)}]{wickens:15}
Mickens RE (2015) \emph{Difference Equations: {Theory}, Applications and
  Advanced Topics} (New York: CRC Press), 3 edition.

\bibitem[{Naor(1969)}]{naor:69}
Naor P (1969) The regulation of queue size by levying tolls.
  \emph{Econometrica} 37(1):15--24.

\bibitem[{Neuts(1981)}]{neuts:81}
Neuts MF (1981) \emph{Matrix-Geometric Solutions in Stochastic Models: An
  Algorithmic Approach}. Johns Hopkins Series in the Mathematical Sciences
  (Johns Hopkins University Press), ISBN 9780801825606.

\bibitem[{Paschalidis \protect\BIBand{} Tsitsiklis(2000)}]{paschalidis:00}
Paschalidis IC, Tsitsiklis JN (2000) Congestion-dependent pricing of network
  services. \emph{IEEE/ACM Transactions on Networking} 8(2):171--184.

\bibitem[{Puterman(1994)}]{puterman:94}
Puterman ML (1994) \emph{Markov Decision Processes} (New York: John Wiley \&
  Sons, Inc.).

\bibitem[{Talluri \protect\BIBand{} {van Ryzin}(2004)}]{tall:04b}
Talluri KT, {van Ryzin} GJ (2004) \emph{The Theory and Practice of Revenue
  Management} (New York: Springer).

\bibitem[{Wang et~al.(2018)Wang, Andrad{\'o}ttir, \protect\BIBand{}
  Ayhan}]{wang:16}
Wang X, Andrad{\'o}ttir S, Ayhan H (2018) Optimal pricing for tandem queues
  with finite buffers. Technical report, Mississippi State University.

\bibitem[{Wolff(1982)}]{wolff:82}
Wolff RW (1982) Poisson arrivals see time average. \emph{Operations Research}
  30(2):223--231.

\bibitem[{Yoon \protect\BIBand{} Lewis(2004)}]{yoon:04}
Yoon S, Lewis ME (2004) Optimal pricing and admission control in a queueing
  system with periodically varying parameters. \emph{Queueing Systems: Theory
  and Applications} 47(3):177--199.

\bibitem[{Ziya et~al.(2006)Ziya, Ayhan, \protect\BIBand{} Foley}]{ziya:06}
Ziya S, Ayhan H, Foley RD (2006) Optimal prices for finite capacity queueing
  systems. \emph{Operations Research Letters} 34(2):214--218.

\bibitem[{Ziya et~al.(2008)Ziya, Ayhan, \protect\BIBand{} Foley}]{ziya:08}
Ziya S, Ayhan H, Foley RD (2008) A note on optimal pricing for finite capacity
  queueing systems with multiple customer classes. \emph{Naval Research
  Logistics} 55(5):412--418.

\end{thebibliography}

%% Here starts the e-companion (EC)
%%%%%%%%%%%%%%%%%%%%%%%%%%%%%%%%%%%%%%%%%%%%%%%%%%%%%%%%%%
\ECSwitch

%\ECDisclaimer
%%%%%%%%%%%%%%%%%%%%%%%%%%%%%%%%%%%%%%%%%%%%%%%%%%%%%%%%%%

%%% Main head for the e-companion
\ECHead{Supporting Proofs}\label{sec:supporting proofs}
This section provides supporting lemmas and supplementary proofs.

\section{Proof of Proposition~\ref{prop:lower bound long-run average actual arrivals}} \label{app:blocking_probability}
In this section, we prove the following proposition:
\begin{repeatproposition}[Proposition~\ref{prop:lower bound long-run average actual arrivals} Revisited]
    For the $M/\hypo/1/B_{1}$-queueing system with a large enough $B_{1}$, we have
    \begin{align*}
        \lim_{t\to\infty} \E\left[\frac{\widetilde{A}(t)}{t}\right]
        ~\geq~ \tilde{\lambda}(1-c\,{p}^{B_{1}-1}),
    \end{align*}
    where $c$ and $p<1$ are positive constants that depend only on $\mu_{1},\mu_{2},\dots,\mu_{J}$, and $\lambda^{\pi}$.
\end{repeatproposition}
Recall that the $M/\hypo/1/B_{1}$-queueing system consists of one buffer and a single server such that:
\begin{enumerate}[label=(\arabic*)]
    \item Arrival process is Poisson with rate $\lambda^{\pi}>0$.
    \item Service time follows the hypo-exponential distribution with $J$ phases and rates $\mu_1,\mu_2,\dots,\mu_J$.
    \item Buffer size is $B_{1}\in\mathbb N\cup\{\infty\}$ so that the maximum number of customers in the system is $B_{1}+1$ when $B_{1}$ is finite.
    \item The system is underloaded (i.e., $\lambda^{\pi}(1/\mu_{1}+1/\mu_{2}+\dots+1/\mu_{J})<1$).
\end{enumerate}
Let $\tilde{A}(t)$ represent the cumulative number of customers who enters the system (i.e., ``actual'' customers) up to time $t$.

To prove the above proposition, we investigate the blocking probability $\tilde{\beta}$ (i.e., the steady-state probability that an incoming customer arrives when the system is saturated and the customer is lost) and show 
\begin{align}\label{eq:upper bound of blocking probability}
      \widetilde{\beta}~\leq~ c\,{p}^{B_{1}-1},
\end{align}
for large enough $B_{1}$, where $c$ and $p<1$ are positive constants that depend only on $\mu_{1},\mu_{2},\dots,\mu_{J}$, and $\lambda^{\pi}$, and 
\begin{align}\label{eq:expected long-run average actual customers}
	\lim_{t\to\infty} \E\left[ \frac{\widetilde{A}(t)}{t} \right]~=~\lambda^{\pi}(1-\widetilde{\beta}),
\end{align}
from which Proposition~\ref{prop:lower bound long-run average actual arrivals} immediately follows.

\subsection{Proof of \texorpdfstring{\eqref{eq:upper bound of blocking probability}}{Lg}: An Upper Bound on the Blocking Probability}
We show \eqref{eq:upper bound of blocking probability} by deriving an upper bound on the blocking probability of $M/\hypo/1/B_{1}$-queueing system. Throughout this section, the norm of vectors or matrices is $1$-norm, which is defined by
\begin{align*}
    \norm{\bx}~\defi~&|x_1|+|x_2|+\cdots+|x_J|\qquad \textrm{for $J$-dimensional vector $\bx\in\mathbb R^{J}$};\\
    \norm{\bR}~\defi~&\max\left\{\frac{\norm{\bR\,\bx}}{\norm{\bx}}\,:\,\bx\in\mathbb R^{J} \right\}\qquad \textrm{for $J\times J$ matrix $\bR\in\mathbb R^{J\times J}$};
\end{align*}
Therefore, we have $\norm{\bR\,\bx}\leq\norm{\bR}\,\norm{\bx}$. Furthermore, for positive square matrix $\bR$, we denote by $\osp(\bR)$ the spectral radius, which is the largest among the absolute values of eigenvalues of $\bR$. Then, Gelfand's formula states that
\begin{align}\label{eq:gelfand}
	\lim_{n\to\infty} \norm{\bR^n}^{1/n}~=~\osp(\bR),
\end{align}
for any positive square matrix $\bR$.

To prove \eqref{eq:upper bound of blocking probability}, we first review the stationary distribution of the $M/\hypo/1/\infty$-queueing system in which the size of the buffer is infinite. Note that the blocking probability of the system is zero.
The state of $M/\hypo/1/\infty$ queueing system is represented by $(n,k)\in \mathbb Z_+\times\{1,2,\dots,J\}$, in which $n$ is the number of customers in the system and $j$ is the phase of current service. The state of no customers is denoted by $(0,0)$.  
We denote by $\widetilde{\boldsymbol\eta}$ the stationary distribution of the system, in which $\widetilde{\eta}(n,j)$ is the probability that the steady-state system has $n$ customers and the phase of service is $j$. 
We also let $\widetilde{\boldsymbol{\eta}}(n,\cdot)$ be a $J$-dimensional vector, of which the $j$-th entry  is $\widetilde{\eta}(n,j)$ for all $n\in\mathbb N$.
% Many researches have been done on the stationary distribution of $M/\hypo/1/\infty$ (e.g., see  ) and 
The following lemma states a matrix-geometric representation on the stationary distribution $\widetilde{\boldsymbol\eta}$. 
\begin{lemma} {\cite[Theorem 3.2.1 and Theorem 3.2.1]{neuts:81}}  \label{lemma:matrix-geometry for Hypo}\ \\ % 
	For the stationary distribution $\widetilde{\boldsymbol\eta}$ of the $M/\hypo/1/\infty$-queueing system, we have
	\begin{align*} %\label{eq:tilde_J_n_1_from_tilde_J_n}
            \widetilde{\boldsymbol{\eta}}(n+1,\cdot)~=~ 
            \bR^{n}\,\widetilde{\boldsymbol{\eta}}(1,\cdot)\qquad\forall n\in\mathbb Z_+,
    \end{align*}
    where $\bR$ is a $J\times J$ positive matrix with spectral radius $\osp(\bR)<1$.
\end{lemma}
\noindent
An explicit formula for $\bR$ can be found in Definition 1 and Theorem 1 of \cite{martin:11}.

Since $\osp(\bR)<1$, Gelfand's formula in \eqref{eq:gelfand} implies that, for any $p$ such that $\osp(\bR)<p<1$ (e.g., we can set $p=(\osp(\bR)+1)/2$), there exists $N$ such that 
\begin{align*}
  \norm{\bR^n}^{1/n}~\leq~p\qquad\textrm{or}\qquad\norm{\bR^n}~\leq~p^{n}\qquad\forall n\geq N.
\end{align*}
Therefore, from Lemma~\ref{lemma:matrix-geometry for Hypo}, we have 
\begin{align}\label{eq:upper bound for eta}
  \norm{\widetilde{\boldsymbol{\eta}}(n+1,\cdot)}~=~\norm{\bR^n\,\widetilde{\boldsymbol{\eta}}(1,\cdot)}
  ~\leq~\norm{\bR^{n}}\, \norm{\widetilde{\boldsymbol{\eta}}(1,\cdot)}
  ~\leq~ p^n\norm{\widetilde{\boldsymbol{\eta}}(1,\cdot)}
\end{align}
for all $n\geq N$.

On the other hand, we express the blocking probability $\tilde{\beta}$ of the  $M/\hypo/1/B_{1}$-queueing system in terms of stationary distribution $\widetilde{\boldsymbol{\eta}}$. Note that $\norm{\widetilde{\boldsymbol{\eta}}(n,\cdot)}=\sum_{j=1}^{J} \widetilde{\eta}(n,j)$ is the probability that there are $n$ customers in the steady-state $M/\hypo/1/\infty$-queueing system.
Then, since $\tilde{\beta}$ is the probability that the $M/\hypo/1/B_{1}$ system has $(B_{1}+1)$ customers in the steady state, we have
\begin{align}\label{eq:blocking_probability}
        \widetilde{\beta}
        ~=~\frac{\norm{\widetilde{\boldsymbol\eta}(B_{1}+1,\cdot)}}{1-\sum_{n=B_{1}+2}^\infty\norm{\widetilde{\boldsymbol\eta}(n,\cdot)}},
\end{align}
which follows from well-established comparison results on queueing systems with finite and infinite buffers (cf.~\cite{cooper:72}).

Combining \eqref{eq:upper bound for eta} and \eqref{eq:blocking_probability}, we obtain that
\begin{align*}
 \widetilde{\beta}~\leq~  
                 \frac{p^{B_{1}}}{1-
                        \frac{p^{B_{1}+1}}{1-B_{1}}\,\norm{\widetilde{\boldsymbol\eta}(1,\cdot)}
                }\,\norm{\widetilde{\boldsymbol\eta}(1,\cdot)} 
                ~\leq~c\,p^{B_{1}}\qquad \forall K\geq N,
\end{align*}
where
   $ c~=~
    \norm{\widetilde{\boldsymbol\eta}(1,\cdot)}
    \left(1-
    \frac{1}{1-p}\,\norm{\widetilde{\boldsymbol\eta}(1,\cdot)}\right)^{-1}
    $. Since $\tilde{\boldsymbol\eta}$ depends only on $\mu_{1},\mu_{2},\dots,\mu_{J}$ and $\lambda^{\pi}$, so do
    $\bR$, $\osp(\bR)$, and $c$, which completes the proof of \eqref{eq:upper bound of blocking probability}.

\subsection{Proof of \texorpdfstring{\eqref{eq:expected long-run average actual customers}}{Lg}}
For stochastic process $\tilde{A}(t)$ --- the cumulative number of actual customers up to time $t$ in the $M/\hypo/1/B_{1}$-queueing system, we prove \eqref{eq:expected long-run average actual customers}:
\begin{align*} %\label{eq:arrival blocking}
	\lim_{t\to\infty} \E\left[ \frac{\widetilde{A}(t)}{t} \right]~=~\lambda^{\pi}(1-\widetilde{\beta}).
\end{align*}
\noindent
From the strong law of large numbers for CTMC and Theorem $1$ in \cite{wolff:82} (i.e., Poisson Arrivals See Time Average: PASTA), we have  
\begin{align*}
    \lim_{t\to\infty} \frac{\widetilde{A}(t)}{N(t)}~=~1-\widetilde\beta\qquad\textrm{almost surely},
\end{align*}
where $N(t)$ is the number of all potential arrivals by time $t$
(i.e.,   $N(t;\bomega):=\max\{m:V(m;\bomega)\leq t\}$).
Since $N(t)$ is a Poisson process with rate $\lambda^\pi$, we have
\begin{align*} %\label{eq:a_s_convergence}
    \lim_{t\to\infty} \frac{\widetilde{A}(t)}{t}~=~
    \lim_{t\to\infty} \frac{N(t)}{t}\frac{\widetilde{A}(t)}{N(t)}~=~
    \lambda^\pi(1-\widetilde\beta)\qquad\textrm{almost surely}.
\end{align*}
Also, since $\tilde{A}(t)\leq N(t)$ almost surely and $\{N(t)/t:t>0\}$ is uniformly integrable, we conclude that $\{\tilde{A}(t)/t:t>0\}$ is uniformly integrable. Therefore, by Vitali's Convergence Theorem (cf.~\cite{Folland:99}), we have that 
\begin{align*}
        \lim_{t\to\infty} \E\left[\frac{\widetilde{A}(t)}{t}\right]~=~\lambda^\pi(1-\widetilde{\beta}),
\end{align*}
which completes the proof.

% Now, we prove the uniform integrality of $\{\widetilde{A}(t)/t:t\geq0\}$.
% For any $x\in\mathbb R_+$ and any $t\geq 0$, $\widetilde{A}(t,\bomega)>x$ implies $N(t,\bomega)>x$. Thus, it follows that
% \begin{align*}  
%         \widetilde{A}(t)\mathbb{I}_{\{\widetilde{A}(t)>xt\}}~\leq~N(t)\mathbb{I}_{\{N(t)>xt\}}\qquad\textrm{almost surely}.
% \end{align*}
% Therefore, we have
% \begin{align*}
%         \E\left[\frac{\widetilde{A}(t)}{t}\mathbb{I}_{\left\{\frac{\widetilde{A}(t)}{t}>x\right\}}\right]
%         ~\leq~
%         \E\left[\frac{N(t)}{t}\mathbb{I}_{\left\{\frac{N(t)}{t}>x\right\}}\right]
%         ~=~\E\left[X\mathbb{I}_{\{X>x\}}\right],
% \end{align*}
% where $X$ is a Poisson random variable with rate $\lambda^\pi$. Because the right-hand side goes to $0$ as $x$ goes to infinity, we  conclude that $\{ \widetilde{A}(t)/t\,:\,t\geq 0\}$ is uniformly integrable. \\
% Because of uniform integrality of $\{ \widetilde{A}(t)/t\,:\,t\geq 0\}$  and \eqref{eq:a_s_convergence}, we have that
% \begin{align*}
%         \lim_{t\to\infty} \E\left[\frac{\widetilde{A}(t)}{t}\right]~=~\lambda^\pi(1-\widetilde{\beta})
% \end{align*}
% by Vitali's Convergence Theorem (cf.~\cite{Folland:99}), which completes the proof.

\section{Proof of Proposition~\ref{prop:A_geq_tilde_A}}\label{sec:proof of coupling}
\label{sec:asymptotic optimality of the optimal st policies}
By constructing sample paths for the actual arrival processes $A(t)$ and $\tilde{A}(t)$ and showing that every \emph{coupled} sample path satisfies $A(t)\geq\tilde{A}(t)$ for all $t\geq 0$, this section provides the proof of the proposition below:
\begin{repeatproposition}[Proposition~\ref{prop:A_geq_tilde_A} Revisited]
	There exists a common probability space, in which $A(t)$ and $\tilde{A}(t)$ are defined such that
\begin{align*}
    A(t)~\geq~\widetilde{A}(t)\quad\textrm{for all $t\geq 0$}
\end{align*}
almost surely, where $A(t)$ and $\tilde{A}(t)$ are the cumulative numbers of actual customers up to time $t$ in the tandem queueing system $(B_1,B_2,\dots,B_J)$  and the $M/\hypo/1/B_1$-queueing system, respectively. 
\end{repeatproposition}

\subsection{Construction of Sample Paths}
To build sample paths of the actual arrival process $A(t)$  for system 
$(B_1,B_2,\dots,B_J)$ under static policy $\pi$, we first introduce fundamental processes for the construction.

\begin{definition}{\bf (Fundamental Processes)}
    \begin{enumerate}[label=(\arabic*)]
    \item
 {\bf The Potential Arrival Time Process}, $\{V(m)\}_{m\in\mathbb N}$, defined on probability space $(\Omega_V,\mathcal{F}_V,\mathbb P_V)$, is a discrete-time process, where $V(m+1)-V(m)$ is independent and exponentially distributed with mean $1/\lambda^{\pi}$ for all $m\in\mathbb Z_+$, where we define $V(0)=0$.
$V(m)$ is the arrival time of the $m$-th ``potential'' customer.
    \item
 {\bf The Service Time Process of Station $j$}, $\{U_{j}(n)\}_{n\in\mathbb N}$, defined on a probability space $(\Omega_{j},\mathcal{F}_{j},\mathbb P_{j})$, is a discrete-time process, where $U_{j}(n)$ is independent and exponentially distributed with mean $1/\mu_{j}$ for $j=1,2,\dots,J$.
$U_{j}(n)$ is the service time of the $n$-th ``actual'' customer at station $j$. 
    \end{enumerate}
\end{definition}
For the rest of Section~\ref{sec:proof of coupling}, we work with the product space 
\begin{align*}
        (\Omega,\mathcal{F},\mathbb P):=(\Omega_V\times\Omega_1\times\Omega_2\times\dots\times\Omega_J,
        \mathcal{F}_V\times\mathcal{F}_1\times\mathcal{F}_2\times\dots\times\mathcal{F}_J,
        \mathbb P_V\times\mathbb P_1\times\mathbb P_2\times\dots\times\mathbb P_J).
\end{align*}
With a slight abuse of notation, we use the same symbols $V(m)$ and  $U_{j}(n)$ for their corresponding extensions on $\Omega$; that is, $V(m;\bomega):=V(m;\omega_V)$ and $U_{j}(n;\bomega):=U_{j}(n;\omega_{j})$, where $\bomega\in\Omega$ and $\bomega=(\omega_V,\omega_1,\omega_2,\dots,\omega_J)$. 

For $n\in\mathbb N$ and $j\in\{1,2,\dots,J\}$, we denote by $T(n,j)$ the time when the $n$-th actual customer leaves station $j$ (e.g., the $n$-th customer leaves the system at $T(n,J)$). 
We also let $T(n,0)$ be the arrival time of the $n$-th actual customer. 
By notational convenience, we set $T(n,j)=0$ for $n\in\mathbb{Z}$ with $n\leq 0$ and $j\in\{1,2,\dots,J\}$.
Since an actual customer is a potential customer who arrives witnessing a non-saturated station~1, %\tdel{and the size of the buffer in the station is $B_1$,} 
the $n$-th actual customer is the first potential customer whose arrival time is greater than the time when the $(n-B_1-1$)-th actual customer leaves station $1$.
Therefore, we have   
\begin{subequations}
\begin{align}
        T(n,0;\bomega)~=~\min\left\{V(m;\bomega)\,:\,V(m;\bomega)>T(n-B_1-1,1;\bomega)\right\} \label{eq:T_n_0}
\end{align}	
for all $n\in\mathbb N$. Also, station $j$ starts serving the $n$-th actual customer immediately after below three conditions are satisfied:
\begin{enumerate}[label=(\arabic*), itemsep=0in]
	\item The $n$-th actual customer is in station $j$;
	\item The $(n-1)$-th actual customer is not in station $j$;
	\item The $(j+1)$-th station is not saturated. In other words, the $(n-B_{j+1}-1)$-th actual customer is not in station $(j+1)$;
\end{enumerate}
Since the precessing time of the job is $U(n,j)$, we have
\begin{align}
        T(n,j;\bomega)~=~\Big(T(n,i-1;\bomega)\vee T(n-1,j;\bomega)\vee T(n-B_{j+1}-1,j+1;\bomega)\Big)+U(n,j;\bomega)  \label{eq:T_n_j}
\end{align}	
for all $n\in\mathbb N$ and $j=1,2,\dots,J$, where we set $T(n,J+1)=0$ for all $n\in\mathbb Z$. From the definition of $T(n,0)$, actual arrival process $A(t)$ becomes
\begin{align*} %\label{eq:A_from_T}
	A(t;\bomega)~=~\max\left\{n\,:\,T(n,0;\bomega)\leq t \right\}
\end{align*}
for all $t\geq 0$. 
\end{subequations}

Now, we define similar random variables for the $M/\hypo/1/B_1$-queueing system as follows:
\begin{enumerate}[label=(\arabic*), itemsep=0in]
	\item $\widetilde{T}(n,0)$ is the arrival time of the $n$-th actual customer;
	\item $\widetilde{T}(n,J)$ is the time when the $n$-th actual customer leaves the system;
	\item $\widetilde{A}(t)$ is the number of customers who actually enter the system  up to time $t$.
\end{enumerate}
Then, this queueing system can be coupled with the tandem queueing system $(B_1,B_2,\dots,B_1)$ in the previous section using fundamental processes $\{V(m)\}_{m\in\mathbb N}$ and $\{U(n,j)\}_{n\in\mathbb N}$ for all $j=1,2,\dots,J$ in that $V(m)$ represents the time when the $m$-th potential customer arrives and $U(n,1)+U(n,2)+\dots+U(n,J)$ is the service time for the $n$-th actual customer. Therefore, we have that
\begin{subequations}
\begin{align}
    \widetilde{T}(n,0;\bomega)~=~&\min\left\{V(m;\bomega)\,:\,V(m;\bomega)>\widetilde{T}(n-1,J);\bomega)\right\}, \label{eq:tilde_T_n_0} \\
    \widetilde{T}(n,J;\bomega)~=~&\widetilde{T}(n,0;\bomega)+U(n,1;\bomega)+U(n,2;\bomega)+\dots+U(n,J;\bomega), \label{eq:tilde_T_n_i} \\
    \widetilde{A}(t;\bomega)~=~&\max\left\{n\,:\,\widetilde{T}(n,0;\bomega)\leq t\right\}, \label{eq:tilde_A_from_tilde_T}
\end{align}
\end{subequations}
where we set $\widetilde{T}(n,0)=\widetilde{T}(n,J)=0$ for $n\leq 0$.

\subsection{Proof of \texorpdfstring{$A(t;\bomega)\geq\tilde{A}(t;\bomega)$}{Lg}}
	We first show, by induction on $n$, that 
	\begin{align}\label{eq:T_leq_tilde_T}
		T(n,0;\bomega)~\leq~\widetilde{T}(n,0;\bomega)\quad\textrm{and}\quad T(n,J;\bomega)~\leq~\widetilde{T}(n,J;\bomega)
	\end{align}
	for all $n\in\mathbb N$ and $\bomega\in\Omega$.
	\begin{enumerate}
		\item[i.] For $n=0$, by \eqref{eq:T_n_0} and \eqref{eq:tilde_T_n_0}, we have $T(1,0;\bomega)=\widetilde{T}(1,0;\bomega)=V(1;\bomega)$. 
		From \eqref{eq:T_n_j}, we also obtain
		\begin{align*}
			T(1,J;\bomega)
			~=~&T(1,J-1;\bomega)+U(1,J;\bomega)\\
			~=~&T(1,J-2;\bomega)+U(1,J-1;\bomega)+U(1,J;\bomega)\\
			~=~&\dots\\
			~=~&T(1,0;\bomega)+U(1,1;\bomega)+U(1,2;\bomega)+\dots+U(1,J;\bomega)~=~\widetilde{T}(1,J;\bomega),
		\end{align*}
		which implies that \eqref{eq:T_leq_tilde_T} holds for $n=1$.
		\item[ii.] Suppose that \eqref{eq:T_leq_tilde_T} holds for $n=1,2,\dots,k$.
                For $n=k+1$, by \eqref{eq:T_n_0} and \eqref{eq:tilde_T_n_0}, we have $T(k+1,0;\bomega)\leq\widetilde{T}(k+1,0;\bomega)$ because $T(k-B_1-1,1;\bomega)\leq T(k,J;\bomega)\leq \widetilde{T}(k,J;\omega)$.
                From \eqref{eq:T_n_j} and \eqref{eq:tilde_T_n_i}, we also have
		\begin{align*}
			T(k+1,J;\bomega)
			~=~&T(k+1,J-1;\bomega)+U(k+1,J;\bomega)\\
			~=~&T(k+1,J-2;\bomega)+U(k+1,J-1;\bomega)+U(k+1,J;\bomega)\\
			~=~&\dots\\
			~=~&T(k+1,0;\bomega)+U(k+1,1;\bomega)+U(k+1,2;\bomega)+\dots+U(k+1,J;\bomega)\\
            ~\leq~&\widetilde{T}(k+1,J;\bomega),
		\end{align*}
        which proves that \eqref{eq:T_leq_tilde_T} also holds for $n=k+1$.
\end{enumerate}
By induction, we have  $T(n,0;\bomega)~\leq~\widetilde{T}(n,0;\bomega)$ and $T(n,J;\bomega)~\leq~\widetilde{T}(n,J;\bomega)$ for all $n\in\mathbb N$ and $\bomega\in\Omega$.

In coupled systems, since $T(n,0;\bomega)\leq \widetilde{T}(n,0;\bomega)$ almost surely, we have
\begin{align*}
        A(t;\bomega)~=~\max\Big\{n:T(n,0;\bomega)\leq t \Big\}
        ~\geq~ \max\left\{n:\widetilde{T}(n,0;\bomega)\leq t\right\}~=~\widetilde{A}(t;\bomega)\qquad\forall \ \bomega\in\Omega,
\end{align*}
which completes the proof of Proposition~\ref{prop:A_geq_tilde_A}.

\end{document}